\documentclass[letterpaper,10pt]{amsart}

\usepackage[all]{xy}                      %
  
\CompileMatrices                            

\UseTips                                    

\input xypic
\usepackage[bookmarks=true]{hyperref}       

\usepackage{amssymb,latexsym,amsmath,amscd}
\usepackage{xspace}
\usepackage{color}
\usepackage{kpfonts}
\usepackage{graphicx}
\usepackage{dsfont}

\reversemarginpar

\theoremstyle{plain}
\newtheorem{theorem}{Theorem}[section]
\newtheorem*{theorem*}{Theorem}
\newtheorem{proposition}[theorem]{Proposition}
\newtheorem{corollary}[theorem]{Corollary}
\newtheorem{lemma}[theorem]{Lemma}

\newtheorem{assumption}[theorem]{Assumption}

\theoremstyle{definition}
\newtheorem{definition}[theorem]{Definition}

\newtheorem{remark}[theorem]{Remark}
\newtheorem{example}[theorem]{Example}

\newcommand{\enm}[1]{\ensuremath{#1}}          %
\newcommand{\op}[1]{\operatorname{#1}}
\newcommand{\cal}[1]{\mathcal{#1}}

\newcommand{\CC}{\enm{\mathbb{C}}}
\newcommand{\EE}{\enm{\mathbb{E}}}

\newcommand{\ZZ}{\enm{\mathbb{Z}}}

\newcommand{\PP}{\enm{\mathbb{P}}}

\newcommand{\MM}{\enm{\mathbb{M}}}
\newcommand{\UU}{\enm{\mathbb{U}}}

\newcommand{\WW}{\enm{\mathbb{W}}}

\newcommand{\Aa}{\enm{\cal{A}}}
\newcommand{\Bb}{\enm{\cal{B}}}
\newcommand{\Cc}{\enm{\cal{C}}}
\newcommand{\Dd}{\enm{\cal{D}}}
\newcommand{\Ee}{\enm{\cal{E}}}
\newcommand{\Ff}{\enm{\cal{F}}}
\newcommand{\Gg}{\enm{\cal{G}}}
\newcommand{\Hh}{\enm{\cal{H}}}
\newcommand{\Ii}{\enm{\cal{I}}}

\newcommand{\Kk}{\enm{\cal{K}}}
\newcommand{\Ll}{\enm{\cal{L}}}
\newcommand{\Mm}{\enm{\cal{M}}}
\newcommand{\Nn}{\enm{\cal{N}}}
\newcommand{\Oo}{\enm{\cal{O}}}

\newcommand{\Rr}{\enm{\cal{R}}}
\newcommand{\Ss}{\enm{\cal{S}}}
\newcommand{\Tt}{\enm{\cal{T}}}

\renewcommand{\phi}{\varphi}
\renewcommand{\theta}{\vartheta}
\renewcommand{\epsilon}{\varepsilon}

\newcommand{\Hom}{\op{Hom}}
\newcommand{\Ext}{\op{Ext}}

\newcommand{\End}{\op{End}}

\renewcommand{\to}[1][]{\xrightarrow{\ #1\ }}

\newcommand{\old}[1]{}

\begin{document}

\title[Logarithmic co-Higgs structure]{Existence of nontrivial logarithmic co-Higgs structure on curves}

\author{Edoardo Ballico and Sukmoon Huh}

\address{Universit\`a di Trento, 38123 Povo (TN), Italy}
\email{edoardo.ballico@unitn.it}

\address{Sungkyunkwan University, Suwon 440-746, Korea}
\email{sukmoonh@skku.edu}

\thanks{The first author is partially supported by MIUR and GNSAGA of INDAM (Italy). The second author is supported by Basic Science Research Program 2015-037157 through NRF funded by MEST and the National Research Foundation of Korea(KRF) 2016R1A5A1008055 grant funded by the Korea government(MSIP)}


\begin{abstract}
We study various aspects on nontrivial logarithmic co-Higgs structure associated to unstable bundles on algebraic curves. We check several criteria for (non-)existence of nontrivial logarithmic co-Higgs structures and describe their parameter spaces. We also investigate the Segre invariants of these structures and see their non-simplicity. In the end we also study the higher dimensional case, specially when the tangent bundle is not semistable. 
\end{abstract}

\maketitle

\tableofcontents

\section{Introduction}
A logarithmic co-Higgs sheaf on a complex manifold $X$ is a pair $(\Ee, \Phi)$ with a torsion-free coherent sheaf $\Ee$ on $X$ and a morphism $\Phi : \Ee \rightarrow \Ee \otimes \Tt_{\Dd}$ satisfying the integrability condition $\Phi \wedge \Phi=0$, where $\Tt_{\Dd}$ is the logarithmic tangent bundle $X$ associated to an arrangement $\Dd$ of hypersurfaces with simple normal crossings. When $\Dd$ is empty, it is a co-Higgs sheaf in the usual sense, introduced and developed by Hitchin and Gualtieri; see \cite{Hit, Gual}. When $\Ee$ is locally free, it is a generalized vector bundle on $X$ considered as a generalized complex manifold, whose co-Higgs field vanishes in the normal direction to the support of $\Dd$. 

It is observed in \cite[Theorem 1.1]{BH} that the semistability of a co-Higgs bundle $(\Ee, \Phi)$ on $X$ with nonnegative Kodaira dimension implies the semistability of $\Ee$. In case of negative Kodaira dimension, there are several works on description of moduli space of semistable co-Higgs bundles, including the case when the associated bundle is not stable; see \cite{Rayan} and \cite{VC1}. 

Now the additional condition for a co-Higgs field to vanish in the normal direction to $\Dd$ with higher degree, forces the associated bundle to be unstable. So we are mainly interested in the logarithmic co-Higgs sheaves associated to the arrangement with high degree and assume that the length of Harder-Narasimhan filtration is at least two. We fix numeric data for the Harder-Narasimhan filtration of the sheaf in consideration, i.e. fix the length $s$ at least two of the filtration together with rank $r_i$ and degree $d_i$ of the successive quotients in the Harder-Narasimhan filtration (\ref{hn}). Setting $\gamma:=\deg \Tt_{\Dd}$ and $\mu_i:=d_i/r_i$, we always assume that $\mu_s-\mu_1 \le \gamma<0$ as the least requirement for the existence of the non-trivial co-Higgs field; see Corollary \ref{aaa2}. Then we investigate the numeric criterion for the sheaf to admit a non-trivial co-Higgs field; see Proposition \ref{e2} and Theorem \ref{g4}. 
\begin{theorem}
Fix the numeric data for the Harder-Narasimhan filtration and denote by $\UU$ the set of the torsion-free sheaves on an algebraic curve $X$ with these data. Then the following hold:
\begin{itemize}
\item [(i)] there exists an unstable sheaf in $\UU$ with non-trivial co-Higgs field;
\item [(ii)] the inequality $\mu_s-\mu_1 \ge \gamma+1-g$ implies the existence of an unstable sheaf in $\UU$ with no non-trivial co-Higgs field;
\item [(iii)] the inequality $\mu_s-\mu_1 < \gamma+1-g$ implies that every sheaf in $\UU$ admits a non-trivial co-Higgs field. 
\end{itemize}
\end{theorem}
The existence part is induced by explicit usage of positive elementary transformations and the positive answer to the Lange conjecture \cite{RT}. Furthermore we extend the notion of Segre invariant to the setting of logarithmic co-Higgs sheaves and show that it is well-defined over curves under the assumption that $\gamma <0$ and that this invariant is same as the usual Segre invariant under a certain condition; see Corollary \ref{h55} and Proposition \ref{yyy}. 
\begin{theorem}
For a logarithmic co-Higgs sheaf $(\Ee, \Phi)$ on an algebraic curve $X$ with $\gamma <0$, the $k^{th}$-Segre invariant $s_k(\Ee, \Phi)$ is well-defined. It is also equal to the Segre invariant $s_k(\Ee)$ in the usual sense, if $\Ee$ admits the complete Harder-Narasimhan filtration, i.e. $r_i=1$ for all $i$. 
\end{theorem}
\noindent Then we check in Proposition \ref{yyy} that co-Higgs sheaves associated to unstable bundle are usually not stable, not even simple. 

Over algebraic curves the bundle $\Tt_{\Dd}$ is automatically semistable. So, as the counterpart to the case of algebraic curves, in \S \ref{Sh} we deal with the case when the dimension of $X$ is at least $2$ and $\Tt_{\Dd}$ is not semistable. Under the assumption that the biggest slope in the Harder-Narasimhan filtration of $\Tt_{\Dd}$ is negative, we give a recipe to construct all the pairs $(\Ee ,\Phi)$ with $\Ee$ reflexive of $\mathrm{rk}(\Ee ) =r\in \{2,3\}$ and non-trivial co-Higgs field $\Phi : \Ee \rightarrow \Ee \otimes \Tt _{\Dd}$. When $r=2$ and in most cases with $r=3$, the map $\Phi$ is always $2$-nilpotent and so it is integrable. We also /point out exactly when we cannot guarantee the integrability.


\section{Preliminary}
Let $X$ be a smooth projective variety of dimension $n$ at least one with the tangent bundle $T_X$ over the field of complex numbers $\CC$. We fix an ample line bundle $\Oo_X(1)$ and denote by $\Ee(t)$ the twist of $\Ee$ by $\Oo_X(t)$ for any coherent sheaf $\Ee$ on $X$ and $t\in \ZZ$. We also denote by $\Ee^\vee$ the dual of $\Ee$. The dimension of cohomology group $H^i(X, \Ee)$ is denoted by $h^i(X,\Ee)$ and we will skip $X$ in the notation, if there is no confusion. We define the slope $\mu(\Ee)$ of a coherent sheaf $\Ee$ on $X$ with respect to $\Oo_X(1)$ to be $\deg \Ee/ \mathrm{rk}(\Ee)$. 

Now consider an arrangement $\Dd=\{D_1, \ldots, D_m\}$ of pairwise distinct, smooth and irreducible divisors $D_i$ on $X$, and if there is no confusion we also denote by $\Dd$ the divisor $D_1+\ldots + D_m$. We assume that the divisor $\Dd$ has simple normal crossings. Then the associated logarithmic tangent bundle $T_X(-\log \Dd)$ is locally free and fits into the following exact sequence; see \cite{D}. 
\begin{equation}\label{log1}
0\to T_X(-\log \Dd) \to T_X \to \oplus_{i=1}^m {\epsilon_i}_*\Oo_{D_i}(D_i) \to 0,
\end{equation}
where $\epsilon_i: D_i \rightarrow X$ is the embedding. If there is no confusion, we will simply denote $T_X(-\log \Dd)$ by $\Tt_{\Dd}$. 

\begin{definition}\cite{BH}
A {\it $\Dd$-logarithmic co-Higgs} sheaf on $X$ is a pair $(\Ee, \Phi)$ where $\Ee$ is a torsion-free coherent sheaf on $X$ and $\Phi: \Ee \rightarrow \Ee \otimes \Tt_{\Dd}$ with $\Phi \wedge \Phi=0$. Here $\Phi$ is called the {\it logarithmic co-Higgs field} of $(\Ee, \Phi)$ and the condition $\Phi \wedge \Phi=0$ is an integrability condition originating in the work of Simpson \cite{S}. 
\end{definition}

For a torsion-free coherent sheaf $\Ee$ on $X$, we consider its associated Harder-Narasimhan filtration:
\begin{equation}\label{hn}
\{0\} =\Ff _0\subset \Ff _1\subset \cdots \subset \Ff _s =\Ee
\end{equation}
with the graduation $gr (\Ee) := \oplus _{i=1}^{s} \Ff _i/\Ff _{i-1}$ such that each $\Ff _i/\Ff _{i-1} $ is semistable and $\mu (\Ff _i/\Ff _{i-1})$ is strictly decreasing for all $i<s$. The integer $s$ is called {\it the length} of the filtration, and if $s=r$, then the filtration is said to be {\it complete}. We denote by $\mu_+(\Ee)$ and $\mu_- (\Ee)$ the maximal and minimal slopes in the filtration, respectively:
$$\mu_+ (\Ee):=\mu (\Ff_1) ~,~ \mu_-(\Ee):=\mu (\Ff_s/\Ff_{s-1}).$$

\begin{remark}\label{aaa0}
For two torsion-free sheaves $\Aa$ and $\Bb$ on $X$, let $\Aa \overline{\otimes}\Bb$ be the quotient of $\Aa\otimes \Bb$ by its torsion. If $\Aa$ and $\Bb$ are semistable, then $\Aa \overline{\otimes}\Bb$ is also semistable by \cite[Theorem 2.5]{Maruyama}. Applying this observation to the Harder-Narasimhan filtrations of $\Aa$ and $\Bb$, we get that $\mu _+(\Aa \overline{\otimes}\Bb ) = \mu _+(\Aa) +\mu _+(\Bb)$.
\end{remark}

\begin{lemma}\label{aaa1}
If $f: \Aa \rightarrow \Bb$ is a nonzero map between two torsion-free sheaves on $X$, then we have $\mu _-(\Aa) \le \mu _+(\Bb)$ 
\end{lemma}

\begin{proof}
Let $\{0\}=\Aa_0 \subset \Aa _1\subset \cdots \subset \Aa _a =\Aa$ be the Harder-Narasimhan filtration of $\Aa$ and let $k\in \{1,\dots ,a\}$ be the minimal integer such that $\Aa _k\nsubseteq \mathrm{ker}(f)$, i.e. the minimal integer such that $f_{|\Aa _k}\not\equiv  0$. Then we have $f_{|\Aa _{k-1}}\equiv 0$ and so $f_{|\Aa _k}$ induces a nonzero map $\widetilde{f}: \Aa _k/\Aa _{k-1}\rightarrow \Bb$. 

Let $\{0\}=\Bb_0 \subset \Bb _1\subset \cdots \subset \Bb _b=\Bb$ be the Harder-Narasimhan filtration of $\Bb$ and let $l$ be the minimal positive integer $l \le b$ such that $\widetilde{f}(\Aa _k/\Aa _{k-1})\subseteq \Bb _l$. Then we have $\widetilde{f} (\Aa _k/\Aa _{k-1})\nsubseteq \Bb _{l-1}$ and so $\widetilde{f}$ induces a nonzero map $\hat{f} : \Aa _k/\Aa _{k-1} \rightarrow \Bb _l/\Bb _{l-1}$. Since  $\Aa _k/\Aa _{k-1}$ and $ \Bb _l/\Bb _{l-1}$ are semistable, we have $\mu ( \Aa _k/\Aa _{k-1})\le \mu (\Bb _l/\Bb _{l-1})$. By the definition of $\mu _+$ and $\mu _-$ in terms of the Harder-Narasimhan filtration we have $\mu ( \Aa _k/\Aa _{k-1})\ge \mu _-(\Aa)$ and $ \mu (\Bb _l/\Bb _{l-1}) \le \mu _+(\Bb)$, concluding the assertion. 
\end{proof}

Remark \ref{aaa0} and Lemma \ref{aaa1} give the following whose assertion will be assumed throughout this article. 
\begin{corollary}\label{aaa2}
Assuming the existence of a nonzero map $\Phi: \Ee \rightarrow \Ee \otimes \Tt _{\Dd}$, we have
\begin{equation}\label{min}
\mu _-(\Ee )\le \mu _+(\Ee) +\mu _+(\Tt_{\Dd})=\mu_+(\Ee \otimes \Tt_{\Dd}).
\end{equation}
\end{corollary}

\begin{remark}\label{p1}
Assume that $\Ee$ is not semistable and so $s\ge 2$. If there exists a nonzero map $f\in \Hom (\Ee /\Ff _{s-1},\Ff _{s-1}\otimes \Tt_{\Dd})$, then we may composite the quotient map $\Ee \rightarrow \Ee /\Ff _{s-1}$ with it to get a nonzero $2$-nilpotent logarithmic co-Higgs field $\Phi_f$. Note that the associated co-Higgs field is uniquely determined by the choice of a map, i.e. if $f$ and $g$ are two different nonzero maps in $\Hom (\Ee /\Ff _{s-1},\Ff _{s-1}\otimes \Tt_{\Dd})$, then we get $\Phi_f \ne \Phi_g$. 
\end{remark}

If $n$ is at least two, we fix a polarization $\Oo _X(1)$ with respect to which we consider (semi-)stability. For most cases in this article we will mainly assume that $\Dd$ is of high degree so that $\Tt_{\Dd}$ is ``sufficiently negative'' and that $\Tt_{\Dd}$ is semistable with $\gamma=\deg \Tt_{\Dd}<0$, except in \S \ref{us};
\begin{assumption}
We always assume that $\gamma=\deg \Tt_{\Dd}$ is negative, if there is no specification.
\end{assumption}

\begin{remark}
There are manifolds with $\Omega _X^1$ ample as in \cite{db1,db2}, in which cases we may even take $\Dd =\emptyset$: If instead of logarithmic co-Higgs field we use the field $T_X(-\Dd)\cong T_X\otimes \Oo_X(-\Dd)$ vanishing on a divisor $\Dd$, then we may use the semistability of the tangent bundle of many Fano manifolds \cite{PW} and then take a very positive $\Dd$ to get $T_X(-\Dd)$ negative and semistable.
\end{remark}

We fix a triple of integers $(r,d,s)\in \ZZ^{\oplus 3}$ together with pairs $( r_i, d_i)\in \ZZ^{\oplus 2}$ for $1\le i \le s$ such that $r\ge 2$, $s\ge 1$, $r_i\ge 1$ and 
$$r =r_1+\cdots +r_s ~,~d = d_1+\cdots +d_s.$$
Assume further that $d_i/r_i > d_{i+1}/r_{i+1}$ for $i=1,\dots ,s-1$. Then we denote by 
$$\UU=\UU_X(s; r_1, d_1;\dots; r_s,d_s)$$ 
the set of all torsion-free coherent sheaves $\Ee$ of rank $r$ such that the Harder-Narasimhan filtration (\ref{hn}) of $\Ee$ with respect to $\Oo _X(1)$ has $(r_1,d_1;\cdots ;r_s, d_s)$ as its numerical data, i.e. each quotient sheaf $\Ff _i/\Ff _{i-1}$ is semistable of rank $r_i$ and degree $d_i$. By \cite{Maruyama} the filtration (\ref{hn}) tensored by $\Tt_{\Dd}$
\begin{equation}\label{hn1}
\{0\}=\Ff_0\otimes \Tt_{\Dd} \subset \Ff _1\otimes \Tt_{\Dd}\subset \cdots \subset \Ff _s\otimes \Tt_{\Dd}
\end{equation}
is the Harder-Narasimhan filtration of $\Ee \otimes \Tt_{\Dd}$ if $\Tt_{\Dd}$ is semistable. We also assume the existence of a nonzero co-Higgs field $\Phi : \Ee \rightarrow \Ee \otimes \Tt_{\Dd}$: if $n$ is at least two, we do not assume for the moment the integrability condition $\Phi \wedge \Phi =0$, because in the most examples in this article it will follow from the other assumption, or from Lemma \ref{p1}, where we assume that $s$ is at least two.

Denote by $\widetilde{\Phi}$ the following map:
$$ \Ee \otimes \Tt_{\Dd} \to \Ee \otimes \Tt_{\Dd}^{\otimes 2}$$
induced by $\Phi$. Comsiting the natural map $\Tt_{\Dd}^{\otimes 2} \rightarrow \wedge^2 \Tt_{\Dd}$ with $\widetilde{\Phi}\circ \Phi$, we have $\Phi\wedge \Phi$ as an element in $\mathrm{Hom}(\Ee, \Ee \otimes \wedge^2 \Tt_{\Dd})$. 

\begin{lemma}\label{p1}
If $s$ is at least two, then we have $\widetilde{\Phi} \circ \Phi =0$, i.e. $\Phi$ is $2$-nilpotent. In particular, we have $\Phi \wedge \Phi =0$.
\end{lemma}

\begin{proof}
Since we assume that $\gamma$ is negative, the sheaf $\Ff _1\otimes \Tt_{\Dd} \subset \Ee \otimes \Tt_{\Dd}$ is the Harder-Narasimhan filtration of $\Ee \otimes \Tt_{\Dd}$
and  $\Ff _1\otimes \Tt_{\Dd}^{\otimes 2} \subset \Ee \otimes \Tt_{\Dd}^{\otimes 2}$ is the Harder-Narasimhan filtration of $\Ee \otimes \Tt_{\Dd}^{\otimes 2}$. Thus we have $\Phi (\Ee )\subseteq \Ff _1\otimes \Tt_{\Dd}$ and $\widetilde{\Phi}(\Ff _1\otimes \Tt_{\Dd})=0$, implying that $\widetilde{\Phi} \circ \Phi =0$.
\end{proof}

\section{Curve case}\label{curve}
Assume that $X$ is a smooth algebraic curve of genus $g$ and take $\Dd=\{p_1, \ldots, p_m\}$ a set of $m$ distinct points. Then we have $\Tt_{\Dd} \cong T_X\otimes \Oo_X(-\Dd)$ with degree $\gamma:=2-2g-m$. We assume that $\gamma$ is negative so that we are not in the set-up of \cite{Nitsure}. The sequence (\ref{log1}) turns into the following 
$$0\to \Tt_{\Dd} \to T_X \to \oplus_{i=1}^m \CC_{p_i} \to 0.$$
Another feature of the case $n=1$ is that all logarithmic co-Higgs fields automatically satisfy the integrability condition.

Consider a vector bundle $\Ee$ of rank $r$ with the Harder-Narasimhan filtration (\ref{hn}) and we assume 
\begin{equation}\label{bseq}
\mu_-(\Ee)=\mu (\Ff _s/\Ff _{s-1})\le \gamma+\mu (\Ff_1)=\gamma+\mu_+(\Ee),
\end{equation}
which is a necessary condition for the existence of a nonzero map $\Phi: \Ee \rightarrow \Ee \otimes \Tt_{\Dd}$; see Corollary \ref{aaa2}. 

\begin{remark}
For each $i\in \{1,\dots ,s\}$ with $\mu (\Ff _i/\Ff _{i-1}) +\gamma \ge \mu_- (\Ee)$, define
$$b(i):=\min_{i+1 \le k \le s} \left\{ k~|~ \mu(\Ff_i/\Ff_{i-1}) + \gamma \ge \mu (\Ff_k/\Ff_{k-1})\right\}$$
and then the map $\Phi$ induces a map $\Phi^i: \Ff_{b(i)}/\Ff_{b(i)-1} \rightarrow \left(\Ff_i/\Ff_{i-1}\right) \otimes \Tt_{\Dd}$. Similarly, for each $j\in \{2,\dots ,s\}$ with $\mu (\Ff _j/\Ff _{j-1}) \le \gamma+\mu_+ (\Ee)$, define
$$c(j):=\max_{1\le k \le j-1} \left\{ k~|~ \mu (\Ff _j/\Ff _{j-1}) \le \gamma+\mu (\Ff _k/\Ff _{k-1})\right\}.$$ 
The map $\Phi$ induces a map $\Phi_j: \Ff _j/\Ff _{j-1} \rightarrow \left(\Ff _{c(j)}/\Ff _{c(j)-1}\right) \otimes \Tt_{\Dd}$. Note that these maps $\Phi^i$ and $\Phi_j$ are not necessarily nonzero. 
\end{remark}

Now fix the following numeric data
$$(s~;~ r_1, \ldots, r_s~;~ d_1, \ldots, d_s)\in \ZZ^{\oplus (2s+1)}$$
with $s,r_i>0$ for each $i$ such that $d_i/r_i> d_{i+1}/r_{i+1}$ for all $i$; if $g=0$, we also assume $a_i/r_i\in \ZZ$ for each $i$. Recall that we denote by $\UU_X (s;r_1,d_1;\dots ;r_s,d_s)$ the set of all vector bundles $\Ee$ of rank $r:=\sum_{i=1}^sr_i$ on $X$ with the Harder-Narasimhan filtration (\ref{hn}) such that $\mathrm{rk}(\Ff _i/\Ff _{i-1}) = r_i$ and $\deg \Ff _i/\Ff _{i-1} =d_i$ for each $i$. The conditions just given above for $s$, $r_i$ and $d_i$ are the necessary and sufficient conditions for the existence of a vector bundle $\Ee$ on $X$ with rank $r$ and degree $d:=d_1+\cdots +d_s$. 

Indeed, for the existence part, in case $g\ge 2$ we may even take a stable bundle $\Ff _i/\Ff _{i-1}$, while in case $g=1$ by Atiyah's classification of vector bundles on elliptic curves, we may take as $\Ff _i/\Ff _{i-1}$ a semistable bundle; we can choose either indecomposable one or polystable one, depending on our purpose. 

To get parameters spaces we first get parameter spaces for the sheaves $\Ee$, then for a fixed sheaf $\Ee$ we study all logarithmic co-Higgs fields $\Phi : \Ee \rightarrow \Ee \otimes \Tt_{\Dd}$ and then we put together the informations. We have several problems coming from the sheaves $\Ee$ such as non-separatedness or often reducibility of moduli of sheaves, and then more problems bring the logarithmic co-Higgs field into the picture.

First of all, we fix enough numerical invariant to get a bounded family of pairs $(\Ee ,\Phi )$. Fixing an ample line bundle $\Oo _X(1)$, we consider sheaves $\Ee$ with a Harder-Narasimhan filtration (\ref{hn}) and we fix the Hilbert function of each subquotient $\Ff _i/\Ff _{i-1}$. Since each $\Ff _i/\Ff _{i-1}$ is assumed to be semistable, the family of all $\Ff _i/\Ff _{i-1}$ are bounded. We first see that the $\Ext ^1$-groups involved in the extensions
$$0\to \Ff _1\to \Ff _2\to \Ff _2/\Ff _1\to 0$$ 
are upper bounded and that the set of all $\Ff _2$ is bounded. Then we consider the set of all $\Ff_3$ and so on, inductively. We may get relative $\Ext ^1$-groups as parameter spaces, but these parameter spaces usually do not parametrizes one-to-one isomorphism classes of sheaves, even by taking into account that proportional extensions gives isomorphic sheaves. 

For the relative $\Ext^1$ we need to have universal family parametrizing all $\Ff _i/\Ff _{i-1}$ and we usually need to work with parameter spaces of sheaves which do not parametrizes one-to-one isomorphic classes. Note that there is a flat family with isomorphic sheaves $\Ee$ whose flat limit is $gr (\Ee ) =\oplus _{i=1}^{s} \Ff _i/\Ff _{i-1}$. Thus there is no hope of one-to-one parametrization of  isomorphism classes of sheaves; when the numerology allows that some $\Ff _i/\Ff _{i-1}$ is strictly semistable, then this phenomenon occurs even for the graded subquotient $\Ff _i/\Ff _{i-1}$. Algebraic stacks of course do not parametrize isomorphism classes of sheaves, not even of vector bundles; see \cite{Gomez}. In the case $n=1$ with $X =\PP^1$, we have a unique bundle, $\Ee$ for any fixed parameter space $\UU_{\PP^1}(s;r_1,d_1;\cdots ;r_s,d_s)$ and so the parameter space for $(\Ee ,\Phi )$ is the vector space $\Hom (\Ee ,\Ee \otimes \Tt _{\Dd})$, which parametrizes one-to-one the isomorphism classes of pairs $(\Ee ,\Phi )$. See Remark \ref{h02.1} for the case $n=1$ and $X$ a curve of genus $g\ge 2$.

\begin{remark}
In the case $s=2$, the datum of $(\Ee ,\Phi )$ with $[\Ee] \in \UU_X (2;r_1,d_1;r_2,d_2)$ and $\Phi :  \Ee \rightarrow \Ee \otimes \Tt_{\Dd}$ induces a holomorphic triple $\psi : \Ee /\Ff _1\to \Ff _1\otimes \Tt_{\Dd}$ in the sense of \cite{bgg} and we may study the stability of the holomorphic triple. Conversely, for every holomorphic triple $f: \Gg _2\rightarrow \Gg _1\otimes \Tt_{\Dd}$ such that $\Gg _1$ and $\Gg _2$ are semistable with $\mathrm{rk}(\Gg _i) = r_i$ and $\deg \Gg _i =d_i$, $i=1,2$, and for any extension class 
\begin{equation}\label{eq+a1}
0 \to \Gg _1\to \Ee \to \Gg _2\to 0,
\end{equation}
we get $[\Ee] \in \UU_X (2;r_1,d_1;r_2,d_2)$ with $0\subset \Gg_1 \subset \Ee$ as its Harder-Narasimhan filtration and a $2$-nilpotent map $\Phi : \Ee \rightarrow \Ee \otimes \Tt_{\Dd}$ induced by $f$. Two sheaves, say $\Ee$ and $\Ee'$, fitting as middle bundles in (\ref{eq+a1}) for the same $\Gg _1$ and $\Gg _2$ are isomorphic if and only if their associated extensions are proportional, because $\Gg _1$ and $\Gg _2$ are assumed to be semistable with $d_1/r_1>d_2/r_2$ and so  (\ref{eq+a1}) is the Harder-Narasimhan filtration of the bundle in the middle.

This argument fits very well in \S 5.1, where $\Tt _{\Dd}$ is assumed to be semistable, because $\Ff _i\otimes \Tt _{\Dd}$ would be in the Harder-Narasimhan filtration of $\Ee\otimes \Tt _{\Dd}$;
in this case we only require that $\Ee$ is torsion-free and then define $\UU (s;r_1,d_1;\dots ;r_s,d_s)$ with Mumford's (slope-)semistability.
\end{remark}

\subsection{Projective line}
We take $X= \PP^1$ and then we have $\Tt_{\Dd} \cong \Oo_{\PP^1}(\gamma)$ with $\gamma<0$. Any vector bundle $\Ee\cong \oplus_{i=1}^r \Oo_{\PP^1}(a_i)$ on $\PP^1$ with $a_1\ge \cdots \ge a_r$ can be rewritten as
\begin{equation}\label{p1b}
\Ee \cong \Oo_{\PP^1}(b_1)^{\oplus r_1}\oplus \cdots \oplus \Oo_{\PP^1}(b_s)^{\oplus r_s},
\end{equation}
with $\sum_{i=1}^s r_i=r$ and $b_1>\cdots > b_s$, i.e. in the Harder-Narasimhan filtration (\ref{hn}) associated to $\Ee$, we have $\Ff_i/\Ff_{i-1}\cong \Oo_{\PP^1}(b_i)^{\oplus r_i}$. Now consider $\UU_{\PP^1} (s;r_1,d_1;\dots ;r_s,d_s)$ with $b_i:= d_i/r_i$ and then it is a single point set, consisting only of $\Ee$. Set 
$$\Delta := \sum _{1\le i<j\le s} \max \{0, \gamma +1+b_i-b_j\}.$$
Then we have $h^0(\mathcal{H}om (\Ee ,\Ee (\gamma ))) =\Delta$. So the parameter space is a well-defined vector space, or its associated projective space if we consider nonzero co-Higgs fields up to scalar multiplication. We have $\Delta >0$ if and only if $b_1+\gamma \ge b_s$. 

For any $\Phi \in \Hom (\Ee ,\Ee (\gamma ))$ and any positive integer $i$, let $\Phi ^{(i)}: \Ee \rightarrow  \Ee (i\gamma )$ be the map obtained by iterating $i$ times a shift of $\Phi$. If $b_1+i\gamma < b_s$ for some $i$, then we have $\Phi ^{(i)} =0$ and so $\Phi$ is a nilpotent logarithmic co-Higgs field. In particular, if $b_1+2\gamma < b_s  \le b_1+\gamma$, then all logarithmic co-Higgs fields are $2$-nilpotents and so we have the following. 

\begin{proposition}
For the bundle $\Ee$ in (\ref{p1b}) with $2\gamma \le b_s-b_1 < \gamma$, the set of its co-Higgs structures is identified with a $\Delta$-dimensional vector space. 
\end{proposition}

Now the assumption in (\ref{bseq}) is simply $b_1+\gamma \ge b_s$ and let $e$ be the last integer $i$ such that $b_i >\gamma+b_1$. Then we may write
\begin{equation}\label{dec}
\Ee \cong \Ee _+\oplus \Ee _-~,~\text{ with }\Ee _+ \cong\oplus _{i=1}^{e} \Oo _{\PP^1}(b_i)^{\oplus r_i} \text{ and }\Ee _- \cong \oplus _{i=e+1}^{r} \Oo _{\PP^1}(b_i)^{\oplus r_i}.
\end{equation}
It is possible to have $e=0$ and so $\Ee_+$ is trivial. Then we have $H^0(\mathcal{E}nd (\Ee )(\gamma)) =H^0(\mathcal{H}om (\Ee_-,\Ee )(\gamma))$. Thus in case of $\PP^1$ we may rephrase our question in the set-up of holomorphic triples $(\Ee_1, \Ee_2, f)$ with $\Ee _1= \Ee _-$, $\Ee _2 = \Ee (\gamma)$ and $f: \Ee_1 \rightarrow \Ee_2$. Here, $\Ee _1$ and $\Ee _2$ are related in a sense that $\Ee _1$ is a twist of a factor of $\Ee _2$. So our general problem concerning nonzero maps $\Phi: \Ee \rightarrow\Ee (\gamma)$ is equivalent to a problem about nonzero maps $\Phi: \Ee _- \rightarrow \Ee (\gamma)$.


\subsection{Elliptic curves}
Let $X$ be an elliptic curve and use the classification of vector bundles on elliptic curves due to M.~Atiyah in \cite{A}. We have $\Tt_{\Dd} \cong \Oo _X(-\Dd)$.

\begin{proposition}\label{e2}
Fix an integer $s\ge 2$ and consider $\UU:=\UU_X (s;r_1, d_1; \dots; r_s, d_s)$ with $d_s/r_s \le d_1/r_1 +\gamma$.
\begin{itemize}
\item [(i)] There exists $[\Ee] \in \UU$ with $\Hom (\Ee ,\Ee (-\Dd )) \ne 0$.
\item [(ii)] If $d_s/r_s = d_1/r_1+\gamma$, there is $[\Ee] \in \UU$ with $\Hom (\Ee ,\Ee (-\Dd )) =0$.
\item [(iii)] If $d_s/r_s < d_1/r_1+\gamma$, then we have $\Hom (\Ee ,\Ee (-\Dd )) \ne 0$ for all $[\Ee] \in \UU$.
\item [(iv)] If $e$ is the maximal integer such that $d_s/r_s \le d_1/r_1+e\gamma$, then we have $\Phi ^{(e+1)} =0$ for every $[\Ee] \in \UU$ and $\Phi \in \Hom (\Ee ,\Ee (-\Dd )) =0$.
\end{itemize}
\end{proposition}

\begin{proof}
Take $[\Ee] \in \UU$ and set $\Ee_s:= \Ff _s/\Ff_{s-1}$. In the set-up of part ({iii}) we have
$\mu (\Ee _s^\vee \otimes \Ff _1(-\Dd) ) >0$ and so Riemann-Roch gives $h^0(\Ee _s^\vee \otimes \Ff _1 (-\Dd)) >0$. Take as $\Phi$ the composition of the surjection $\Ee \rightarrow \Ee _s$ with a nonzero map $\Ee _s \rightarrow \Ff _1(-\Dd)$ and then the inclusion $\Ff _1(-\Dd)\hookrightarrow \Ee(-\Dd)$, proving (iii). 

Now assume $d_s/r_s = d_1/r_1+\gamma$. Take as $\Ee _i$ any semistable bundle with prescribed numeric data so that $\Ee _1$ and $\Ee_s$ are polystable and no factor of $\Ee _1(-\Dd )$ is isomorphic to a factor of $\Ee _s$. Set $\Ee:= \oplus _{i=1}^{s} \Ee _i$. Due to the slope, we have $\Hom (\Ee _i,\Ee _j(-\Dd )) =0$ if $(i,j) \ne (s,1)$. Since every nonzero map between stable bundles with the same slope is an isomorphism, we have $\Hom (\Ee _s,\Ee _1(-\Dd)) =0$ and so $\Hom (\Ee ,\Ee (-\Dd )) =0$, proving part (ii). 

Under the same situation, set $t:=\gcd (|d_s|, r_s)$ and write $r_s=at$ and $d_s=bt$. Then each indecomposable factor of $\Ee_s$ has rank $a$ and degree $b$, which is also stable. Pick one of these indecomposable factors, say $\Aa$. Now from $d_s/r_s=d_1/r_1+\gamma$, we see that $a$ divides $r_1$. Then we have $r_1/a\in \ZZ$ and it also divides $d_1$, say $r_1=ap$ and $d_1=qp$. We also see that $\gcd(a,q)=\gcd(a,b)=1$ and so $\Ee_1$ is a polystable bundle whose factors have rank $a$ and degree $q=b-a\gamma$. Let $\Gg$ be any polystable vector bundle of rank $r_1$ and degree $d_1$ with $\Aa\otimes \Oo_X(\Dd)$ as one of its factors. Set $\Ff := \Gg\oplus \left(\oplus _{i=2}^{s} \Ee _i\right)$ and then we have $[\Ff] \in \UU$. Since $\Hom (\Ee _s,\Gg (-\Dd)) \ne 0$, we have $\Hom (\Ff ,\Ff (-\Dd)) \ne 0$, proving part (i).

Part (iv) is obvious.
\end{proof}

\begin{remark}\label{e3}
In parts (i) and ({iii}) of Proposition \ref{e2} the proof gives the existence of a nonzero $2$-nilpotent co-Higgs field $\Phi$.
\end{remark}


\subsection{Higher genus case}
Assume that $X$ has genus $g\ge 2$. Note that $\gamma \le 2-2g$. For the pairs of integers $(r,d)$ with $r>0$, denote by $\MM _X(r,d)$ the moduli space of the stable vector bundles of rank $r$ on $X$ with degree $d$. It is known to be a non-empty, smooth and irreducible quasi-projective variety of dimension $r^2(g-1) +1$. 

Fix a point $p\in X$ and take any exact sequence on $X$:
\begin{equation}\label{eqg1}
0 \to \Aa \stackrel{u}{\to} \Bb \to \CC _p\to 0
\end{equation}
with $\Aa$ and $\Bb$ locally free. Note that $\mathrm{rk}(\Aa )=\mathrm{rk}(\Bb )$ and that $\deg \Bb =\deg \Aa +1$. Then we say that $\Bb$ is obtained from $\Aa$ by applying {\it a positive elementary transformation at $p$} and that $\Aa$ is obtained from $\Bb$ by applying {\it a negative elementary transformation at $p$}. For a fixed $\Aa$ (resp. $\Bb$) the set of all extensions (\ref{eqg1}) is parametrized by a vector space of dimension $\mathrm{rk} (\Aa)$ (resp. $\mathrm{rk}(\Bb)$); since it is an irreducible variety, we may speak about the general positive elementary transformation of $\Aa$ (resp. a general negative elementary transformation of $\Bb$).

\begin{lemma}\label{g1}
For $(r,d,k)\in \ZZ^{\oplus 3}$ with $r,k>0$, fix a general bundle $[\Aa]\in \mathbb{M}_X(r,d)$. If $\Bb$ is obtained from $\Aa$ by applying $k$ positive elementary transformations, then it is stable.
\end{lemma}

\begin{proof}
Since the statement is trivial for $r=1$, we may assume $r\ge 2$. 

\quad (a) First assume $k=1$ with the sequence (\ref{eqg1}) and that $\Bb$ is not stable so that there exists a subbundle $\Gg_t\subset \Bb$ of rank $t\in \{1,\dots ,r-1\}$ with $\deg \Gg _t/t\ge (d+1)/r$. Let $\Cc\subset \Aa$ be the saturation of $u^{-1}(\Gg)$ and set $a:= \deg \Cc $. Then we have
$a\ge \deg u^{-1}(\Gg ) \ge \deg \Gg -1$. Since $\Aa$ is general, we get by \cite[Theorem 3.10]{l2} or \cite[Theorem 2]{b} that $\mu (\Aa/\Cc ) -\mu (\Cc )\ge g-1$, from which we get 
$$\frac{d}{r} -\frac{a}{t} \ge \frac{(r-t)(g-1)}{r}.$$
Using this with $\deg (\Gg_t) \le a+1$, we get
\begin{align*}
\mu(\Bb)-\mu(\Gg_t)&\ge \frac{d+1}{r} -\frac{a+1}{t} \\
&\ge\frac{t-r +(r-t)(g-1)}{rt}\\
&=\frac{(r-t)(g-2)}{rt}\ge 0,
\end{align*}
The equality holds if and only if $g=2$ and $\deg \Gg_t=a+1$. Let $\widetilde{a}$ be the maximal degree of a rank $t$ subbundle of $\Aa$. For arbitrary $t$ and $g$, Mukai and Sakai proved in \cite{ms} that $td-\widetilde{a}r \le t(r-t)g$, while the quoted results also said that $td-\widetilde{a}r\ge t(r-t)(g-1)$. The precise value of $\widetilde{a}$ is known by an unpublished result of A. Hirschowitz in \cite{Hi} and \cite[Remark 3.14]{l2}, which says that $td-\widetilde{a}r = t(r-t)(g-1)+\epsilon$, where $\epsilon $ is the only integer such that $0\le \epsilon < r$ and $\epsilon +t(r-t)(g-1) \equiv td \pmod{r}$.

Now assume $g=2$. We conclude unless $\epsilon = 0$, $a =\widetilde{a}$ and $\deg \Gg _t =a+1$. In this case we use that we take a general positive elementary transformation of $\Aa$. Since $\epsilon =0$ and $\Aa$ is general, $\Aa$ has only finitely many rank $t$ subbundles of maximal degree $a=\widetilde{a}$, say $\Nn _i$ for $1\le i \le \delta$; see \cite{ox} and \cite{tb}. The fiber $\Nn _{i|\{p\}}$ of $\Nn _i$ at $p$ is a $t$-dimensional linear subspace of the fiber $\Aa _{|\{p\}}$ of $\Aa$ at $p$, which is an $r$-dimensional vector space. The union of these $t$-dimensional linear subspaces $\Nn _i$ for $1\le i \le \delta$, is a proper subset of $\Aa _{|\{p\}}$. Thus, for a general positive elementary transformation $\Bb$ of $\Aa$ at $p$, the saturation $\Mm _i$ of $\Nn _i$ is just $\Nn _i$ for all $i$, i.e. $\deg u^{-1}(\Mm _i) = \deg \Mm _i$ for all $i$, contradicting the assumptions $a =\widetilde{a}$ and $\deg \Gg _t =a+1$.

\quad (b) Now assume $k\ge 2$. The case $k=1$ proves that a general positive elementary transformation of a stable bundle is stable. Similarly a general negative transformation of a stable
bundle is also stable, and so we may apply the step (a) $k$ times to get the assertion. 
\end{proof}

\begin{theorem}\label{g4}
Fix an integer $s\ge 2$ and consider $\UU:=\UU_X (s;r_1, d_1; \dots; r_s, d_s)$ with $d_s/r_s \le d_1/r_1 +\gamma$.
\begin{itemize}
\item [(i)] There exists $[\Ee] \in \UU$ with $\Hom (\Ee ,\Ee\otimes \Tt_{\Dd}) \ne 0$.
\item [(ii)] If $d_s/r_s \ge  d_1/r_1+\gamma +1-g$, there is $[\Ee] \in \UU$ with $\Hom (\Ee ,\Ee\otimes \Tt_{\Dd}) =0$.
\item [(iii)] If $d_s/r_s < d_1/r_1+\gamma +1-g$, then $\Hom (\Ee ,\Ee \otimes \Tt_{\Dd}) \ne 0$ for all $[\Ee] \in \UU$.
\end{itemize}
\end{theorem}

\begin{proof}
Take $[\Ee] \in \UU$ and set $\Ee_s:= \Ff _s/\Ff_{s-1}$. Since $\Ee _s$ and $\Ff _1$ are semistable, the bundle $\Ee _s^\vee \otimes \Ff _1 (-D)$ is also semistable. In the set-up of part ({iii}) we have $\mu (\Ee _s^\vee \otimes \Ff _1\otimes \Tt_{\Dd}) >g-1$ and so Riemann-Roch gives $h^0(\Ee _s^\vee \otimes \Ff _1 \otimes \Tt_{\Dd}) >0$. Take as $\Phi$ the composition of the
surjection $\Ee \rightarrow \Ee _s$ with a nonzero map $\Ee _s \rightarrow \Ff _1\otimes \Tt_{\Dd}$ and then the inclusion $\Ff _1\otimes \Tt_{\Dd}\hookrightarrow \Ee\otimes \Tt_{\Dd}$, proving (iii). 

Now assume the set-up of (ii) and pick a general element $(\Ee _1,\dots ,\Ee _s)$ in 
$$\MM _X(r_1,d_1)\times \cdots \times \MM _X(r_s,d_s).$$
In particular, each $\Ee _i$ is a general stable bundle in $\MM_X(r_i, d_i)$. Set $\Ee := \oplus _{i=1}^{s} \Ee _i$ and then it is sufficient to prove the following claim for (ii). 

\quad {\emph {Claim 1:}} We have $\Hom (\Ee ,\Ee \otimes \Tt_{\Dd}) =0$.

\quad {\emph {Proof of Claim 1:}} Since $\Ee := \oplus _{i=1}^{s} \Ee _i$, it is enough to prove that $H^0(\Ee _i,\Ee _j\otimes \Tt_{\Dd}) =0$ for all $i, j$. We have $H^0(\Ee _i,\Ee _i\otimes \Tt_{\Dd}) =0$ for each $i$, because $\Ee _i$ is stable and $\gamma <0$. Now assume $i\ne j$. Note that $(E_i^\vee,E_j\otimes \Tt_{\Dd})$ is a general element of 
$$\MM _X(r_i,-d_i)\times \MM _X(r_j,d_j+r_j\gamma ).$$ 
We have $\mu (\Ee _i^\vee\otimes \Ee_j \otimes \Tt_{\Dd}) =-\mu (\Ee _i) +\mu (\Ee _j)+\gamma \le g-1$. By a theorem of A.~Hirschowitz in \cite[Theorem 1.2]{RT}, we have $h^0(\Ee _i^\vee \otimes \Ee_j\otimes \Tt_{\Dd}) =0$, concluding the proof of {\it Claim 1}.

Now we prove part (i). Let $\Bb_i$ be a semistable bundle of rank $r_i$ on $X$ with degree $d_i$ for each $i$, and let $\Ee:=\oplus_{i=1}^s \Bb_i$. Our strategy is to find appropriate $\Bb_1$ and $\Bb_s$ with the additional condition $\Hom (\Bb _s,\Bb _1\otimes \Tt_{\Dd})\ne 0$, which would imply part (i). 

\quad (a) Assume $r_s < r_1$. Setting $r':= r_1-r_s$ and $d':= d_1+\gamma r_1-d_s$, it is enough to show the existence of an exact sequence of vector bundles on $X$:
\begin{equation}\label{eqq1}
0 \to \Aa _1\to \Aa _2\to \Aa _3\to 0,
\end{equation}
with $\Aa _1,\Aa _2$ semistable and $\Aa _1$ of rank $r_s$ and degree $d_s$, $\Aa _2$ of rank $r_1$ and degree $d_1+\gamma r_1$. Then $\Aa_3$ would be of rank $r'$ and degree $d'$, and we may take $\Bb _s:= \Aa _1$ and $\Bb _1:= \Aa _2\otimes \Tt_{\Dd}^\vee$. 

Note that for a quadruple of integers $(x_1, x_2, a_1, a_2)\in \ZZ^{\oplus 4}$ with $x_2>x_1>0$ and $a_1/x_1 \le a_2/x_2$ (resp. $a_1/x_1 < a_2/x_2$), we have
$$\frac{a_2}{x_2} \le \frac{a_2-a_1}{x_2-x_1} ~\left(\text{resp. } ~\frac{a_2}{x_2} < \frac{a_2-a_1}{x_2-x_1} \right).$$
Using the above to $(x_1, x_2, a_1, a_2)=(r_s, r_1, d_s, d_1+\gamma r_1)$, together with 
$$\mu_-(\Ee)=d_s/r_s\le d_1/r_1 + \gamma=\mu_-(\Ee\otimes \Tt_{\Dd}),$$
we have $d_1/r_1+\gamma \le d'/r'$ with equality if and only if $d_s/r_s=d_1/r_1 + \gamma$. When the equality holds, we may take as $\Aa _1$ and $\Aa _3$ arbitrary semistable bundles with the prescribed ranks and degrees and then take $\Aa _2:= \Aa _1\oplus \Aa _3$. Now assume $d_s/r_s< d_1/r_1 +\gamma$ and so $d_1/r_1+\gamma < d'/r'$. In this case by the positive answer to the conjecture of Lange, there is an exact sequence (\ref{eqq1}) of vector bundles on $X$ with the prescribed ranks and degrees and with stable $\Aa _1$, $\Aa _2$ and $\Aa _3$; see \cite[Introduction]{RT}.

\quad (b) Assume $r_s>r_1$. Similarly as in (a) we set $r'':= r_s-r_1$ and $d'':= d_1-\gamma r_1 -d_s$. By taking $\Bb _s:= \Aa _2$ and $\Bb _1:= \Aa _3\otimes \Tt_{\Dd}^\vee$, it is sufficient to find an exact sequence (\ref{eqq1}) with $\Aa _2$ and $\Aa _3$ semistables, $\Aa _1$ of rank $r''$ and degree $d''$, $\Aa _2$ of rank $r_s$ and degree $d_s$ and $\Aa _3$ or rank $r_1$ and degree $d_1+\gamma r_1$. 

First assume $d_s/r_s= d_1/r_1 + \gamma$. In this case we have $d''/r''=d_s/r_s$ and we take as $\Aa_1$ and $\Aa _3$ arbitrary semistable bundles with prescribed ranks and degrees and set $\Aa _2:= \Aa _1\oplus \Aa _3$. Now assume $d_s/r_s<d_1/r_1 + \gamma$ and so $d_1/r_1+\gamma > d''/r''$. Again by the conjecture of Lange proved in \cite{RT} we may take $\Aa _1$, $\Aa _2$, $\Aa _3$ with the prescribed ranks and degree and stable.

\quad (c) Assume $r_s =r_1$. First assume $d_s/r_s=d_1/r_1 +\gamma$, i.e. $d_1 = d_s-\gamma r_1$. In this case we take as $\Bb _s$ any semistable bundle with rank $r_s$ and degree $d_s$ and set $\Bb _1:= \Bb _s\otimes \Tt_{\Dd}$. Now assume $k:= d_1 +r_1\gamma -d_s>0$. We take as $\Bb _s$ a general stable bundle of rank $r_s$ and degree $d_s$. Then $\Bb _s\otimes \Tt_{\Dd}$ is a general element of $\MM _X(r_1,d_1-t)$. We take as $\Bb _1$ a bundle obtained from $\Bb _s\otimes \Tt_{\Dd}$ by applying $k$ general positive elementary transformations.  
\end{proof}

\begin{remark}
Consider a smooth algebraic curve $X$ of an arbitrary genus $g\ge 0$ and assume $s\ge 3$ together with 
$$d_2/r_2+\gamma < d_s/r_s\le d_1/r_1+\gamma.$$
By Theorem \ref{g4} in case $g\ge 2$ and Proposition \ref{e2} for $g=1$, there is $(\Ee ,\Phi )$
with $[\Ee] \in \UU_X (s;r_1,d_1;\dots ;r_s,d_s)$ and a nonzero map $\Phi : \Ee \rightarrow \Ee \otimes \Tt _{\Dd}$. Take any $[\Ee] \in \UU_X (s;r_1,d_1;\dots ;r_s,d_s)$ with the Harder-Narasimhan filtration (\ref{hn}). Note that we have $\Hom (\Aa ,\Bb )=0$ for any semistable bundles $\Aa$ and $\Bb$ with $\mu (\Aa )>\mu (\Bb)$. {\it Claim 1} in the proof of Theorem \ref{g4} applied to $\Ee /\Ff _1$ shows that any map $\Phi : \Ee \rightarrow \Ee \otimes \Tt _{\Dd}$ is uniquely determined by $f: \Ee /\Ff _{s-1} \rightarrow \Ff _1\otimes \Tt_{\Dd}$; moreover we get $\mathrm{Im}(\Phi )=\mathrm{Im}(f)$ and $\mathrm{ker}(\Phi )$ is the inverse image of $\mathrm{ker}(f)$ under the surjection $\Ee \rightarrow \Ee /\Ff _{s-1}$. 

Conversely, for $1\le i \le s$, choose arbitrary semistable bundles $\Ee _i$ with $\mathrm{rk}(\Ee _i) =r_i$ and $\deg \Ee _i =d_i$, and a map $f: \Ee _s\rightarrow  \Ee _1\otimes \Tt _{\Dd}$. To get a vector bundle $[\Ee]\in \UU_X (s;r_1,d_1;\dots ;r_s,d_s)$, we only need to consider $(s-1)$ extension classes
$$0 \to \Ff _i \to \Ff _{i+1}\to \Ee _{i+1}\to 0$$ 
for $i=1,\dots ,s-1$, where $\Ff_1:=\Ee_1$. Once $\Ee$ is chosen, the map $\Phi : \Ee \rightarrow \Ee \otimes \Tt _{\Dd}$ is uniquely determined by $f$.
\end{remark}

In case of curves, we sometimes may improve Remark \ref{aaa1} to a strict inequality in the following way. 

\begin{remark}\label{aaa2}
Take $[\Ee] \in \UU:=\UU_X (s;r_1, d_1; \dots; r_s, d_s)$ with the Harder-Narasimhan filtration (\ref{hn}). Assume that $s\ge 2$ with $r _s\ne r_1$, 
$$\gcd (r_1,d_1) =\gcd (r_s,d_s)=1 \text{ and } d_1/r_1+\gamma \ge d_s/r_s.$$ 
Since $\gcd (r_1,d_1)=1$, the sheaf $\Ff _1$ is stable and so $\Ff _1\otimes \Tt_{\Dd}$ is stable. Since $\gcd (r_s,d_s)=1$, the sheaf $\Ff _s/\Ff _{s-1}$ is also stable. From $r_1\ne r_s$ we get that  $\Ff _1\otimes \Tt_{\Dd}$ and $\Ff _s/\Ff _{s-1}$ are not isomorphic and so we have $\Hom (\Ff _s/\Ff _{s-1},\Ff _1\otimes \Tt_{\Dd})=0$. If $s\ge 3$, we obviously have
$\Hom (\Ff _i/\Ff _{i-1},\Ff _j\otimes \Tt_{\Dd})=0$ for all $i,j \in \{1,\dots ,s\}$ with $(i,j) \ne (s,1)$, even without the assumptions $r _s\ne r_1$ and $\gcd(r_1,d_1) =\gcd(r_s,d_s)=1$.
Thus we have $\Hom (\Ee, \Ee\otimes \Tt_{\Dd})=0$.
\end{remark}

\begin{remark}\label{h1}
In the following exact sequence
\begin{equation}\label{eqh1}
0 \to \Aa \to \Cc \to \Bb \to 0
\end{equation}
of vector bundles on $X$ with $\Aa$ and $\Bb$ semistable, if we have $\mu (\Aa) +2-2g > \mu (\Bb)$, then we have $h^1(\Aa \otimes \Bb ^\vee )=0$ and so (\ref{eqh1}) splits. Thus if we have
$$\frac{d_i}{r_i}+2-2g > \frac{d_{i+1}}{r_{i+1}}$$
for all $i$, then we have $\Ee \cong gr (\Ee )$ for all $[\Ee] \in \UU:=\UU_X (s;r_1, d_1; \dots; r_s, d_s)$. 

Now assume $d_i/r_i +2-2g\ge d_{i+1}/r_{i+1}$ for all $i$. If the equality holds for some $i$, i.e. $d_i/r_i = d_{i+1}/r_{i+1}+2g-2$, then we have $r_i\ne r_{i+1}$ and that $r_h$ and $d_h$ are coprime, where $h\in \{i,i+1\}$ is the index with higher rank $r_h =\max \{r_i,r_{i+1}\}$. As in Remark \ref{aaa2} we get $\Ee \cong gr (\Ee )$ for all $[\Ee] \in \UU$.

For example, take $s=2$. We just proved that $\Ee \cong \Ff _1\oplus \Ff _2/\Ff _1$ for all $[\Ee] \in \UU (2;1,d_1;1,d_2)$ with a nonzero map $\Phi :  \Ee \rightarrow \Ee \otimes \Tt_{\Dd}$ and either $\Dd \ne 0$ or $r_1\ne r_2$, and $d_h,r_h\in \ZZ$, where $r_h = \max \{r_1,r_2\}$.
\end{remark}

\begin{example}\label{h2}
Assume $d_1>d_2-\gamma$. For a fixed $\Rr \in \mathrm{Pic}^{d_2}(X)$, consider the set 
$$\EE:=\{ (\Ff _1,\psi)~|~\Ff _1\in \mathrm{Pic}^{d_1}(X) \text{ and }0\ne \psi :\Rr \rightarrow \Ff _1\times \Tt _{\Dd}\}/\sim,$$
where the equivalent relation $\sim$ is given by $(\Ff _1,\psi )\sim (\Ff _1,c\psi )$ for all $c\in \CC^\ast$. $\EE$ is the set of all effective divisors of $X$ with degree $d_1+\gamma -d_2$ and so $\EE$ is isomorphic to a symmetric product of $d_1+\gamma -d_2$ copies of $X$ and in particular it is irreducible. By Remark \ref{h1} we have $\Ee \cong \Ff _1\oplus \Ff _2/\Ff _1$ for all $[\Ee]  \in \UU_X (2;1,d_2+\gamma;1,d_2)$.
\end{example}

\begin{example}\label{h02.1}
Assume $g\ge 2$ and take $\Dd =\emptyset$ so that $\gamma =2-2g$. Fix any $d\in \ZZ$ and consider $\Ee \in \UU_X (2;1,d+2g-2;1,d)$ with a nonzero map $\Phi : \Ee \rightarrow \Ee \otimes T_X$. Set $\Rr := \Ff _2/\Ff _1\in \mathrm{Pic}^d(X)$ and then $\Phi$ is induced by a nonzero map $\psi : \Rr \rightarrow \Ff _1\otimes T_X$. Since $\Ff _1$ is in $\mathrm{Pic}^{d+2g-2}(X)$, the map $\psi$ is an isomorphism. Thus we get $\Ff _1\cong \Rr \otimes \omega _X$ and that for a fixed $\Ee$ the set of all nonzero map $\Phi$ is parametrized by a nonzero scalar. From $h^1(\omega _X)=1$ we see that there are, up to isomorphism, exactly two vector bundles $\Ee$ fitting into an exact sequence
\begin{equation}\label{eqh02.1}
0 \to \Rr\otimes \omega _X \to \Ee \to \Rr \to 0,
\end{equation}
that is, $\left(\Rr \otimes \omega _X\right)\oplus \Rr$ and an indecomposable bundle. Thus the set of all $(\Ee ,\Phi)$, up to isomorphisms, with nonzero $\Phi$, is parametrized one-to-one by the disjoint union of two copies of $\mathrm{Pic}^d(X)\times \CC^\ast$. Thus no one-to-one parameter space is irreducible. We get another irreducible parameter space that is not one-to-one, by taking as parameter space, up to a nonzero constant, the relative $\Ext^1$ group of (\ref{eqh02.1}) over $\mathrm{Pic}^d(X)$; each indecomposable bundle $\Ee$ appears $\infty^1$-times and it has $gr (\Ee )\cong \left(\Rr \otimes \omega_X \right) \oplus \Rr$ as its limit inside the parameter space.
\end{example}

Now for $s$ at least two let us define the set $\UU^{\mathrm{co}}=\UU_X^{\mathrm{co}}(s;r_1, d_1; \ldots ; r_s, d_s)$ of certain co-Higgs bundles associated to $\UU=\UU_X(s;r_1, d_1; \ldots; r_s, d_s)$ as follows. 
\begin{align*}
\UU^{\mathrm{co}}:=\{ (\Ee, \Phi) ~|~ [\Ee] \in &\UU \text{ and } \Phi : \Ee \rightarrow \Ee \otimes \Tt_{\Dd} \text{ with } \mathrm{Im}(\Phi) \subseteq \Ff_1 \text{ and } \mathrm{ker} (\Phi) \subseteq \Ff_{s-1}\\
&\text{ such that } \mathrm{rk} (\mathrm{Im}(\Phi))=\min \{r_1, r_s\} \text{ and } \Ff_1, \Ff_s/\Ff_{s-1} \text{ stable }\}
\end{align*}
Denote by $\Gamma\subseteq \MM _X(r_1,d_1)\times \MM _X(r_s,d_s)$ the set of all pairs $(\Ff _1,\Ff _s/\Ff _{s-1})$ obtained from $\UU^{\mathrm{co}}$ and call the projection from $\Gamma$ to each factor by $\pi_1$ and $\pi_2$, respectively. 

\begin{proposition}\label{h3}
Assume that
\begin{itemize}
\item each $d_i$ is positive such that $d_i/r_i > d_{i+1}/r_{i+1}$ for all $i$, and 
\item $d_1/r_1+\gamma \ge d_s/r_s$.
\end{itemize}
Then we have the following assertions. 
\begin{itemize}
\item [(i)] If $r_1=r_s$, then $\pi _1$ and $\pi _2$ are dominant.
\item [(ii)] If $r_1<r_s$ (resp. $r_1>r_s$) and $d_1/r_1+\gamma > d_s/r_s$, then $\pi _1$ (resp. $\pi _2$) is dominant.
\item [(iii)] Assume $d_1/r_1+\gamma > g-1+ d_s/r_s$. Then $\Gamma$ contains a non-empty open subset of $\MM _X(r_1,d_1)\times \MM _X(r_s,d_s)$; if $r_1\ge r_s$, then we have $\mathrm{ker}(\Phi ) =\Ff _{s-1}$.
\end{itemize}
\end{proposition}

\begin{proof}
Fix a point $([\Aa_1], [\Aa_s]) \in \MM_X(r_1, d_1)\times \MM_X(r_s, d_s)$. For arbitrary $[\Aa_i]\in \MM_X(r_i, d_i)$, $i=2,\cdots, s-1$, we consider $\Ee := \oplus _{i=1}^{s} \Aa _i$ with the Harder-Narasimhan filtration (\ref{hn}) such that $\Ff _1\cong\Aa _1$ and $\Ff _s/\Ff _{s-1} \cong\Aa _s$. We take a map $\Phi : \Ee \rightarrow \Ee \otimes \Tt _{\Dd}$ with $\Phi (\Ff _{s-1})=0$, which is induced by a map $\psi: \Aa _s \rightarrow \Aa _1\otimes \Tt _{\Dd}$ whose existence is guaranteed by the assumptions. 

First assume $r_s=r_1$. We need to prove the existence of a map $\psi$ of rank $r_1$ when $\Aa _1$ is general in $\MM _X(r_1,d_1)$ and $\Aa _s$ is general in $\MM _X(r_s,d_s)$; we do not claim here that $([\Aa _1],[\Aa _2])$ is general in $\MM _X(r_1,d_1)\times \MM_X(r_s,d_s)$. The dominance of $\pi _2$ is the content of Lemma \ref{g1}. The dominance of $\pi _1$ can be proved by applying the dual map, or with the same proof as in the proof of Lemma \ref{g1}, concluding part (i). 

Now assume $r_s < r_1$ and $d_1/r_1+\gamma > d_s/r_s$. We take as $\Aa _s$ a general element of $\MM _X(r_s,d_s)$. The existence of a stable $\Aa _1\in  \MM _X(r_1,d_1)$ with an embedding $\psi : \Aa _s\hookrightarrow \Aa _1\otimes \Tt _{\Dd}$ with $\Aa _1/\psi (\Aa _s)$ stable and general in $\MM _X(r_s-r_1,d_1-\gamma r_1)$ is proved in part (i) of the proof of Theorem \ref{g4}.

By using part (ii) of the proof of Theorem \ref{g4} instead of part (i), we get the case $r_s > r_1$ and $d_1/r_1+\gamma > d_s/r_s$. 

Now consider part (iii) and assume $d_1/r_1+\gamma > g-1 +d_s/r_s$. Take a general $([\Aa _1],[\Aa _s])\in \MM _X(r_1,d_1)\times \MM _X(r_s,d_s)$ and set $\Bb _1:= \Aa _1\otimes \Tt _{\Dd}$. Then it is sufficient to find $\psi : \Aa _s\rightarrow \Bb _1$ with $\mathrm{Im}(\psi ) =\min \{r_1,r_s\}$. By the assumptions, we have $\mu (\Aa _s^\vee \otimes \Bb _1)>g-1$ and so Riemann-Roch gives $\Hom (\Aa _s,\Bb _1) \ne 0$. Take a general element $\psi \in \Hom (\Aa _s,\Bb _1) $ and then it is sufficient to prove that $\psi$ has rank $\min \{r_1,r_s\}$. Note that we have $h^1(\Aa _s^\vee \otimes \Bb _1) =0$ and so $\Bb_1$ is an element of the following set 
$$\WW:= \{[\Ff] \in \MM _X(r_1,d_1+\gamma r_1)\mid h^1(\Aa _s^\vee \otimes \Ff) =0\}.$$
By Riemann-Roch, we have the following, for each $[\Ff]\in \WW$,
$$h^0(\Aa _s^\vee \otimes \Ff) =\deg (\Aa _s^\vee \otimes \Ff) +r_1r_2(1-g) = r_1r_2\left(\frac{d_1}{r_1}+\gamma -\frac{d_s}{r_s}+1-g\right)>0.$$ 
Now take the relative $\Hom$ with $\WW$ as its parameter space, i.e. for each $[\Ff] \in \WW$, the fibre is $\Hom (\Aa _s,\Ff )$. The total space $\Lambda$ of this relative $\Hom$ is irreducible, because $h^0(\Aa _s^\vee \otimes \Ff)$ is constant for all $[\Ff]\in \WW$ by \cite{Banica} and \cite{lange}. By \cite[part (d) of Theorem 1.2]{br}, a general element $(\phi: \Aa _s\rightarrow\Ff )$ of $\Lambda$ has $\phi$ with rank $\min \{r_1,r_s\}$. When $r_1\ge r_s$, this implies that $\mathrm{ker}(\Phi )=\Ff _{s-1}$, because the map $\psi : \Ff _s/\Ff _{s-1}\rightarrow \Ff _1\otimes \Tt _{\Dd }$ is injective if and only if it has rank $r_s$ .
\end{proof}

\begin{remark}\label{h00}
As in the end of proof of Proposition \ref{h3}, to show that the set of the co-Higgs bundles $(\Ee ,\Phi)$ with certain properties is parametrized by an irreducible variety, it sometimes works to prove that (a) the set of all bundles $\Ee$ is parametrized by an irreducible variety $Y$, and (b) the integer $k:= \dim \Hom (\Ee _y,\Ee _y\otimes \Tt _{\Dd})$ is constant for all $y\in Y$. In this case, the set of all $(\Ee ,\Phi )$ with no restriction on $\Phi$ is parametrized by a vector bundle of rank $k$ on $Y$.
\end{remark}

\begin{example}\label{h01}
Assume $r_1=r_s$ and $d_1/r_1+\gamma = d_s/r_s$. Consider a bundle $[\Ee] \in \UU_X (s;r_1,d_1;\dots ;r_s,d_s)$ with an arbitrary map $\Phi : \Ee \rightarrow \Ee \otimes \Tt _{\Dd}$. Since $d_i/r_i > d_1/r_1+\gamma$ for all $i<s$, there is no nonzero map $\Ff _i/\Ff _{i-1}\rightarrow \Ee\otimes  \Tt _{\Dd}$ and so we have $\Ff _{s-1}\subseteq \mathrm{ker}(\Phi)$. On the other hand, since $d_s/r_s > d_j/r_j+\gamma$ for all $j>1$, we have $\Phi (\Ee )\subseteq \Ff _1$. We have $\mathrm{rk}(\Phi )=r_1$ if and only if $\Ff _s/\Ff _{s-1}\cong  \Ff _1\otimes \Tt _{\Dd}$ and $\Phi$ is induced by an isomorphism $\Ff _s/\Ff _{s-1} \rightarrow \Ff _1\otimes \Tt _{\Dd}$.
\end{example}

\begin{example}\label{h02}
Assume $r_1=r_s$ and that
$$d_2/r_2 +\gamma < d_s/r_s \text{ and } d_{s-1}/r_{s-1}>d_1/r_1+\gamma.$$
If we choose $(\Ee ,\Phi )$ with $[\Ee] \in \UU_X (s;r_1,d_1;\dots ;r_s,d_s)$, then as in Example \ref{h01} we see that $\Phi (\Ee )\subseteq \Ff _1\otimes \Tt_{\Dd}$ and $\Ff _{s-1}\subseteq \mathrm{ker}(\Phi)$, because $d_i/r_i+\gamma <d_s/r_s$ for all $i>1$ and $d_j/r_j >d_1/r_1+\gamma$ for all $j<s$. Set $k:=d_1-d_2+\gamma r_1$. Then we have $\mathrm{rk}(\Phi )=r_1$ if and only if $\Ff _1\otimes \Tt _{\Dd}$ is obtained from $\Ff _s/\Ff _{s-1}$ by applying $k$ positive elementary transformations and $\Phi$ is induced by the associated inclusion $\Ff _s/\Ff _{s-1} \hookrightarrow \Ff _1\otimes \Tt _{\Dd}$.
\end{example}

\section{Segre invariant}
In this section, we do not assume that $\Tt _{\Dd}$ has some kind of negativity, so that we may have stable $(\Ee ,\Phi )$ with nonzero $\Phi$. Let $\Ee$ be a torsion-free sheaf of rank $r\ge 2$ and $\Phi: \Ee \rightarrow \Ee \otimes \Tt_{\Dd}$ a co-Higgs field. For a fixed integer $k\in \{1,\dots, r-1\}$, let us denote by $\Ss(k,\Ee, \Phi)$ the set of all subsheaves $\Aa \subset \Ee$ of rank $k$ such that $\Phi (\Aa )\subseteq \Aa \otimes \Tt _{\Dd}$. Define the $k^{th}$-{\it Segre invariant} to be 
$$s_k(\Ee ,\Phi ):= k\deg \Ee -\max _{\Aa \in \Ss (k,\Ee ,\Phi)} r\deg \Aa.$$
In case $\Phi=0$, we simply denote it by $s_k(\Ee)$. This is an extension of the Segre invariant, introduced in \cite{l1} with the notation $s_k(\Ee)$, to the case $n\ge 2$. Over curves this notion was used in several literatures, including \cite{br, b, BL, Hi, l2, ox, RT, tb}. If we take $\Tt _{\Dd}^\vee$ instead of $\Tt _{\Dd}$, we get a definition for logarithmic Higgs fields. Note that we always have $\Ss (k,\Ee ,0)\ne \emptyset$ and $s_k(\Ee) \le s_k(\Ee, \Phi)$. 

\begin{lemma}\label{s1}
Let $(\Ee ,\Phi )$ be a $2$-nilpotent co-Higgs bundle of rank $r$, and set $\Aa := \mathrm{ker}(\Phi)$ and $\Bb := \mathrm{Im}(\Phi)$ with $r' := \mathrm{rk}(\Aa )$. Then we have the following:
\begin{itemize}
\item [(i)] $\Aa \in \Ss (r' ,\Ee ,\Phi )$;
\item [(ii)] $\Ss (k,\Aa ,0) \subseteq \Ss (k,\Ee ,\Phi )$ for $1\le k < r'$;
\item [(iii)] $\Bb$ is torsion-free and $\Phi ^{-1}(\Gg )\in \Ss (k,\Ee ,\Phi)$ for all $\Gg \in \Ss (k-r',\Bb ,0)$ and $r'<k<r$;
\item [(iv)] $\Ss (k,\Ee ,\Phi )\ne \emptyset$ for all $k$.
\end{itemize}
\end{lemma}

\begin{proof}
Parts (i) and (ii) are obvious. Part (iii) is true, because $\Bb \subseteq \Aa\otimes \Tt _{\Dd}$ by the definition of $2$-nilpotent. Part (iv) follows from the other ones.
\end{proof}

\begin{example}
From \cite[Theorem 1.1]{BH1} we get a description of the set of nilpotent co-Higgs structures on a fixed stable bundle of rank two on $\PP^n$. Indeed it is either trivial or an $(n+1)$-dimensional vector space, depending on the parity of the first Chern class and an invariant $x_{\Ee}$. We get a non-trivial set of nilpotent co-Higgs structures on $\Ee$ if and only if $c_1(\Ee)+2x_{\Ee}=-3$, and in this case we get $s_1(\Ee, \Phi)=1$. 
\end{example}


\subsection{Curve case}
From now on we assume $n=1$ with $g=g(X)$ and $\gamma <0$. Take $(\Ee ,\Phi )$ with $[\Ee] \in \UU_X (s;r_1,d_1;\dots ;r_s,d_s)$ and let (\ref{hn}) be the Harder-Narasimhan filtration of $\Ee$.

\begin{remark}\label{h03}
From the assumption $\gamma <0$, we have $\Phi (\Ff _i)\subseteq \Ff _{i-1}\otimes \Tt _{\Dd}$.
\end{remark}
Remark \ref{h03} immediately proves the following two lemmas.

\begin{lemma}\label{h04}
For an integer $j\in \{1,\dots ,s-1\}$, set $k(j)= \sum _{i=1}^{j} r_i$. Then
$$s_{k(j)}(\Ee ,\Phi ) \le  k(j)\deg \Ee  -r\deg \Ff _j$$ 
\end{lemma}

\begin{remark}
We expect that the inequality in Lemma \ref{h04} is in fact equality, although we give the positive answers only to some special cases; see Lemma \ref{h05} and Proposition \ref{h07}. 
\end{remark}

\begin{lemma}\label{h05}
Assume $g=0$ and take $\Ee \cong \oplus _{i=1}^{r} \Oo _{\PP^1}(a_i)$ with $a_i\ge a_j$ for all $i\le j$. Then we have
$$s_k(\Ee ,\Phi )=s_k(\Ee ) =k(a_1+\cdots +a_r) -r(a_1+\cdots +a_k).$$
\end{lemma}

\begin{proposition}\label{h06}
Fix $[\Ee] \in \UU_X(s;r_1,d_1;\dots ;r_s,d_s)$ with $s\ge 2$ and $\Phi \in \Hom (\Ee ,\Ee \otimes \Tt _{\Dd})$. Choose any $k\in \{r_1+1,\dots ,r-r_s+1\}$ such that there is $h\in\{1,\dots ,s-1\}$ with $r_1+\cdots +r_h < k < r_1+\cdots +r_{h+1}$, and set 
\begin{align*}
r'&:= k-r_1-\cdots -r_s,\\
e&:= s_k(\Ff _{h+1}/\Ff _h),\\
d'&:= r'\deg \Ff _{h+1}/\Ff _h -er_{h+1}.
\end{align*}
Let $\Bb\subset \Ff _{h+1}/\Ff _h$ be any subsheaf of rank $r'$ and degree $d'$, and set $\Aa :=u^{-1}(\Bb )$, where $u$ is the surjection in the exact sequence
$$0\to \Ff _h \to \Ff _{h+1} \stackrel{u}{\to} \Ff _{h+1}/\Ff _h\to 0.$$
Then $\Bb \in \Ss (k,\Ee ,\Phi)$ and $s_k(\Ee ,\Phi ) \le k\deg \Ee -k(\deg \Ff _h +e)$.\end{proposition}

\begin{proof}
Note that $d'$ is the degree of all rank $r'$ maximal degree subsheaves of $\Ff _{h+1}/\Ff _h$. Since $\deg \Aa =d'$, we have $\deg \Bb  =\deg \Ff _h+d'$. Since $\Phi (\Ff _{h+1})\subset \Ff _h\subset \Bb$, we have $\Bb \in \Ss (k,\Ee ,\Phi)$. Since $\deg \Bb  =\deg \Ff _h+\deg \Aa$, we get the assertion. 
\end{proof}

Now Lemma \ref{h05} and Proposition \ref{h06} prove the following result.

\begin{corollary}\label{h55}
The Segre invariant $s_k(\Ee ,\Phi)$ is defined for all $(\Ee ,\Phi)$, if $\gamma <0$.
\end{corollary}

Example \ref{h08} shows that in Proposition \ref{h06} we may have strict inequality; of course, to be in the set-up of Proposition \ref{h06} we need to have $r_{h+1}\ge 2$.

\begin{example}\label{h08}
Assume $g\ge 5$ and fix $h\in \{1,\dots ,s-2\}$ with $s\ge 3$. Set $r_{h+1}=r_{h+2}= 2$ and fix $r_i>0$ for $i\notin \{h+1,h+2\}$ and $d_i\in \ZZ$, $i=1,\dots ,s$ such that 
\begin{itemize}
\item $d_i/r_i> d_{i+1}/r_{i+1}$ for all $i=1,\dots ,s-1$ and 
\item $d_{h+1} =2d_{h+2}+1$. 
\end{itemize}
By a theorem of Nagata there is a stable bundle $\Ee _{h+1}$ of rank $2$ with degree $d_{h+1}$ and $g-1 \le s_1(\Ee _{h+1}) \le g$. Here, $s_1(\Ee _{h+1})$ is the only integer $t$ with $g-1\le t\le g$ and $d_{h+1}-t$ even. For $i\ne h+1$ we choose $\Ee _i$ to be any semistable bundle of degree $d_i$
and rank $r_i$. Set $\Ee := \oplus _{i=1}^{s} \Ee _i$ and then we have $\Ff _i= \oplus _{j=1}^{i} \Ee _j$ in the Harder-Narasimhan filtration (\ref{hn}) of $\Ee$. 

Take any $\Phi : \Ee \rightarrow \Ee \otimes \Tt _{\Dd}$ with $\mathrm{ker}(\Phi )\supseteq \oplus _{i=0}^{h+1} \Ee _i$, e.g. take $\Phi =0$ or, for certain $\Ee _1$ and $\Ee _s$ so that there is a nonzero map $\Ee _s\rightarrow \Ee _1\otimes \Tt _{\Dd}$, take a $2$-nilpotent map $\Phi$ with $\mathrm{ker}(\Phi ) \supseteq \Ff _{s-1}$. Let $\Aa \subset \Ff _{h+1}/\Ff _h$ be a line subbundle of maximal degree and then we have $\Bb := u^{-1}(\Aa)= (\sum _{i=1}^{h} \Ee _i)\oplus \Aa$. If we set $\Bb _1:= (\sum _{i=1}^{h} \Ee _i)\oplus \Ee _{h+2}$, then we have $\deg \Bb _1>\deg \Bb$. 
\end{example}

\begin{proposition}\label{h07}
Assume $r_i= 1$ for all $i$. For an integer $k\in \{1,\dots, r-1\}$ and any co-Higgs bundle $(\Ee ,\Phi )$ with $[\Ee] \in \UU_X (r;1,d_1;\dots ;1,d_r)$, we have
$$s_k(\Ee ) = s_k(\Ee ,\Phi) = k\deg \Ee -r\deg \Ff _k$$ 
and $\Ff _k$ is the only bundle achieving the minimum degree in $\Ss (k,\Ee ,\Phi )$.
\end{proposition}

\begin{proof}
By Remark \ref{h03} and Lemma \ref{h04}, we have $[\Ff _k]\in  \Ss (k,\Ee ,\Phi )$. Thus it is sufficient to prove that $\Ff _k$ is the only one achieving the minimum degree in $\Ss (k,\Ee ,0)$. Take any $[\Gg] \in \Ss (k,\Ee ,0)$ with maximal degree. The maximality condition on $\deg \Gg$ implies that $\Ee /\Gg$ has no torsion and so it is a vector bundle of rank $r-k$ on $X$. We use double induction on $k$ and $r$. The case $k=1$ is obvious, because $\Ff _1$ is the first step of the Harder-Narasimhan filtration of $\Ee$. 

Assume that $k$ is at least two and the proposition holds for trivial co-Higgs fields with any $k'\in \{1,\dots ,k-1\}$ and any bundles $\Ee '$ whose Harder-Narasimhan filtration has rank one bundles as subquotients. 

Assume for the moment $\Ff _1\subset \Gg$. Since $\Ff _1$ is saturated in $\Ee$, i.e. $\Ee /\Ff _1$ has no torsion, $\Ff _1$ is saturated in $\Gg$ and $\Gg /\Ff _1$ is a rank $k-1$ subsheaf of the vector bundle $[\Ee /\Ff _1]\in \UU_X (r-1;1,d_2;\dots ;1,d_r)$. The inductive assumption gives $\deg \Gg /\Ff _1\le \deg \Ff _k/\Ff _1$, with equality if and only if $\Gg /\Ff _1 \cong \Ff _k/\Ff _1$, i.e. $\deg \Gg \le \deg \Ff _k$ with equality if and only if $\Gg \cong\Ff _k$.

Now assume $\Ff _1\nsubseteq \Gg$. Since $\Gg$ is saturated in $\Ee$, this means that $\Ff _1+\Gg$ has rank $k+1$. Let $\Nn$ be the saturation of $\Ff _1+\Gg$ in $\Ee$, and then we have $\deg \Nn \ge \deg \Ff _1 + \deg \Gg$ and $\Nn /\Ff _1$ is a rank $k$ subsheaf of $\Ee /\Ff _1$. If $r\ge k-2$, then by the inductive assumption on $r$ we have $\deg \Nn /\Ff _1\le \deg \Ff _{k+1}/\Ff _1 < \deg \Ff _k-\deg \Ff _1$ and so $\deg \Gg < \deg \Ff _k$, a contradiction. Thus we may assume $k=r-1$ and so $\Nn \cong\Ee$. Since $\Ff _1+\Gg$ has rank $k+1$, the natural map $\Gg \rightarrow \Ee /\Ff _1$ is injective. Thus we have $\deg \Gg \le \deg \Ee -\deg \Ff _1< \deg \Ff _{r-1}$, a contradiction.
\end{proof}

Now for $k\in \{1,\dots, r-1\}$ set 
$$ \delta _k(\Ee, \Phi ):= \max _{\Aa \in \Ss (k,\Ee ,\Phi)} \deg (\Aa ),$$
$\delta _0(\Ee ,\Phi ):= 0$ and $\delta _r(\Ee ,\Phi ):= \deg (\Ee)$. In case $\Phi=0$, we simply denote it by $\delta _k(\Ee)$.

\begin{proposition}\label{h012}
Fix $h\in \{1,\dots, s\}$ with $s\ge 2$ and set $\rho:=r_1+\cdots + r_h$. For $[\Ee] \in \UU_X (s;r_1,d_1;\dots ;r_s,d_s)$, we have
\begin{itemize}
\item [(i)] $s_\rho (\Ee ) = \rho \deg \Ee  -k\deg \Ff _h$ and $\Ff _h\subset \Ee$ is the only subsheaf of rank $\rho$ with degree $\deg \Ff _h$;
\item [(ii)] $\deg \Ee \le \Ff _{h-1} + (k-\rho )d_h/r_h$ for $k$ with $\rho -r_h < k < \rho$ and $[\Gg] \in \Ss (k,\Ee ,0)$;
\item [(iii)] $\delta_k(\Ee)-\deg (\Ff_{h-1})$ for $k$ with $\rho -r_h < k < \rho$, equals 
$$\max \{ \sum _{j=h}^{s} \delta _{t_j}(\Ff _j/\Ff _{j-1})~\big|~t_h+\cdots +t_s =k+r_h-\rho \text{ with } 0 \le t_j \le r_j \text{ for all }j\}.$$
\end{itemize}
\end{proposition}

\begin{proof}
Set $\mu _i:= \mu(\Ff_i/\Ff_{i-1})=d_i/r_i$ for $i=1,\dots ,s$ and let $\Gg \subseteq \Ee$ be a rank $\rho$ subsheaf of maximal degree. Then part (i) is trivial if $s=h$, because $\Gg \cong\Ee$ in this case. Thus we may assume that $h< s$. Set $a_0=0$ and 
$$a_i:= \mathrm{ker}(\Ff _i\cap \Gg) \text{ with }k_i:= a_i -a_{i-1},$$
for $i=1,\dots ,s$. If we denote by $\Rr _i\subseteq \Ff _i/\Ff _{i-1}$ the image of $\Ff _i\cap \Gg$ by the quotient map $\pi _i: \Ff _i\rightarrow\Ff _i/\Ff _{i-1}$, then $\Rr _i$ is trivial, i.e. $\Ff _i\cap \Gg \subseteq \Ff _{i-1}$, if and only if $k_i=0$. Setting $S:= \{i\in \{1,\dots ,s\}\mid k_i>0\}$, we have $\sum _{i=1}^{s} k_i=\sum _{i\in S} k_i=\rho$ and that $\Gg \cong \Ff _h$ if and only if $k_i=r_i$ for all $i\le h$, or equivalently $k_i=0$ for all $i>h$. Since $\Ff _0$ is trivial, we have $\Rr _1 \cong\Ff _1\cap \Gg$. Thus we have $\deg \Gg  =\sum _{i\in S} \deg \Rr _i$. Since each $\Ff _i/\Ff _{i-1}$ is semistable, we have $\deg \Rr _i\le k_i\mu _i$ for all $i\in S$ and so we may use that $\mu _i>\mu _j$ for all $i<j$ to get part (i). 

For part (ii) let $\Gg\subset \Ee$ be a rank $k$ subsheaf of maximal degree and define $k_i$, $S\subseteq \{1,\dots ,s\}$ and the sheaves $\Rr_i \subset \Ff _i/\Ff _{i-1}$ as above. Then we have $\sum _{i\in S} k_i =k$ and $\deg \Gg\le \sum _{i\in S} k_i\mu _i$ and again we may use that $\mu _i>\mu _j$ for all $i<j$, to get the assertion. Part (iii) comes directly from the definition of $\delta _k(\Ee)$.
\end{proof}

As immediate corollaries of Theorem \ref{h012} we get the following.

\begin{corollary}\label{h013}
We have $s_k(\Ee ,\Phi ) =s_k(\Ee )=s_k(gr (\Ee ))$ for all $k$.
\end{corollary}

\begin{corollary}\label{h014}
For $k$ with $r -r_s < k < r$, we have 
$$\delta _k(\Ee ,\Phi )=\delta _k(\Ee )= d_1+\cdots +d_{s-1} +\delta _{k+r_s-r}(\Ee /\Ff _{s-1}).$$
\end{corollary}


\subsection{Simplicity}
Again let $X$ be a smooth curve of genus $g$. Fix $\Rr \in \mathrm{Pic}(X)$ and set $\gamma := \deg \Rr$. For a map $\Phi : \Ee \rightarrow \Ee \otimes \Rr$, set
$$\End (\Ee ,\Phi ):= \{f\in \End (\Ee )\mid \hat{f}\circ \Phi = \Phi \circ f\},$$ 
where $\hat{f}$ is the map $f\otimes \mathrm{id}_{\Rr}: \Ee\otimes \Rr \rightarrow \Ee \otimes \Rr$. 

In case $\gamma >0$, it often happens that $\End (\Ee ,\Phi )$ is properly contained in $\End (\Ee )$ and $(\Ee ,\Phi )$ is simple with $\Ee$ not simple, e.g. stable Higgs fields when $g\ge 2$
or stable co-Higgs fields when $g=0$. In this short section, we consider the case $\gamma < 0$ and show why this is seldom the case for $\gamma<0$. 

We assume that $[\Ee] \in \UU_X (s;r_1,d_1;\cdots ;r_s,d_s)$ with the Harder-Narasimhan filtration (\ref{hn}) and that $\Phi : \Ee \rightarrow \Ee \otimes \Rr$ is nonzero and so $s\ge 2$. Note that every endomorphism of $\Ee$ preserves the Harder-Narasimhan filtration of $\Ee$. By Remark \ref{h03}, every endomorphism of $(\Ee ,\Phi )$ also preserves the Harder-Narasimhan filtration of $\Ee$. Now set $\Kk := \mathrm{ker}(\Phi)$ and then we have $\Kk \supseteq \Ff _1$ by the case $i=1$ of Remark \ref{h03} or Lemma \ref{n1} below.
For two maps $\phi \in \End (\Ee /\Ff _{r-1})$ and $\psi \in \Hom (\Ee /\Ff _{r-1},\Kk )$, define a map $f: \Ee \rightarrow \Ee$ to be the following composition:
$$\Ee \to \Ee/\Ff_{r-1} \to \Kk \hookrightarrow \Ee,$$
where the first map is the natural quotient and the second map is given by $\psi\circ\phi$. By the definition of $\Kk$, we have $\Phi \circ f =0$. If $\Phi$ is $2$-nilpotent, i.e. $\mathrm{Im}(\Phi)\subseteq \Kk \otimes \Rr$, e.g. if $s=2$ by Lemma \ref{s1}, we have $\hat{f}\circ \Phi =0$. So, if $\Phi$ is $2$-nilpotent and $\Hom (\Ee/\Ff _{r-1},\Kk )\ne 0$, then we have $\End (\Ee ,\Phi )\not\cong\CC$. We also see from $\Ff _1\subseteq \Kk$ that if $\Hom (\Ee /\Ff _{r-1},\Ff _1)\ne 0$, then we have $\End (\Ee ,\Phi )\not \cong \CC$. By Riemann-Roch, we get $\Hom (\Ee /\Ff _{r-1},\Ff _1)\ne 0$, if $d_s/r_s < d_1/r_1+g-1$. Since $\Phi \ne 0$ and each $\Ff _i/\Ff _{i-1}$ is semistable, we have $d_s/r_s \le d_1/r_1+\gamma$. Now if $\Rr \cong \Tt _{\Dd}$, then we have $\gamma \le 2-2g$ and so $d_s/r_s < g-1 +d_1/r_1$
for all $g\ge 2$. Thus we get the following. 
\begin{proposition}\label{yyy}
For $[\Ee]\in \UU_X(s; r_1, d_1;\dots;r_s,d_s)$ with a nonzero co-Higgs field $\Phi$ on a smooth curve $X$ of genus $g\ge 2$, the pair $(\Ee, \Phi)$ is not simple. 
\end{proposition}

\begin{remark}
In our set-up, adding a nonzero map $\Phi$ to an unstable bundle $\Ee$ does not help enough to get
a semistable pair $(\Ee,\Phi )$; usually it is not simple, e.g. any endomorphism inducing $\Ff_s \rightarrow \Ff_1$ commutes with $\Phi$.
\end{remark}


\section{Higher dimensional case}\label{Sh}
In this section we consider the case when the dimension of $X$ is at least two. Note that a coherent sheaf $\Ee$ on $X$ is semistable if and only if $\mu_+(\Ee)=\mu_-(\Ee)$.

\subsection{Case of $\Tt_{\Dd}$ semistable}\label{sss}
We fix a polarization $\Oo _X(1)$ with respect to which we consider slope, stability and semistability. We assume that $\Tt_{\Dd}$ is semistable with $\mu (\Tt_{\Dd}) < 0$; in case $\mu (\Tt_{\Dd})\ge 0$, we would get that the framework would be the construction of stable or semistable co-Higgs or logarithmic co-Higgs bundles as in \cite{BH}. There are several manifolds $X$ with $T_X$ semistable, or equivalently with the semistable cotangent bundle; see \cite{PW}. 

Choose a pair $(\Ee ,\Phi)$ with $\Ee$ a torsion-free sheaf of rank $r$ and $\Phi : \Ee \rightarrow\Ee\otimes \Tt_{\Dd}$ with the Harder-Narasimhan filtration (\ref{hn}) of $\Ee$. Then the sheaf $\Ff _i/\Ff _{i-1}$ is a torsion-free semistable sheaf for all $i$ and $\mu (\Ff _i/\Ff _{i-1}) > \mu (\Ff _{i-1}/\Ff _{i-2})$ for every $i>1$. As in \S \ref{curve} on curve case, for fixed integers $r_i$ and $d_i$, we consider the set $\UU_X(s;r_1, d_1, \dots, r_s, d_s)$ of torsion-free sheaves of rank $r$ on $X$ with the Harder-Narasimhan filtration (\ref{hn}) with subquotient $\Ff_i/\Ff_{i-1}$ of ranks $r_i$ and degrees $d_i$ for $i=1,\dots, s$. 

Recall that in characteristic zero the tensor product of two semistable sheaves is still semistable by \cite[Theorem 2.5]{Maruyama}, So if $\Phi$ is not trivial, then we get $s\ge 2$ and so $\Ee$ is not semistable with the Harder-Narasimhan filtration (\ref{hn1}) for $\Ee\otimes \Tt_{\Dd}$. If $\Aa$ is a semistable torsion-free sheaf, then we have 
$$\mu (\Aa \otimes \Tt _{\Dd}) = \mu (\Aa )+\gamma.$$
Thus if $[\Ee] \in \UU_X (s;r_1,d_1;\dots ;r_s,d_s)$ and there is a nonzero map $\Phi :  \Ee \rightarrow \Ee\otimes \Tt_{\Dd}$, then we get $d_1/r_1+\gamma \ge d_s/r_s$; see Corollary \ref{aaa2}. Now let us use the same idea in Lemma \ref{aaa1}. Define 
$$\ell_2=\ell_2(\Ee):=\max_{1\le i \le s}\{~i~|~\mu (\Ff _i/\Ff_{i-1}) + \gamma \ge \mu (\Ff _s/\Ff _{s-1})\}$$ 
and then we have $\Phi(\Ee )\subset \Ff _{\ell_2}\otimes \Tt_{\Dd}$. From $\gamma <0$, we get $\ell_2 \le s-1$. On the other hand, letting  
$$\ell_1=\ell_1(\Ee):=\min_{1 \le j \le s}\{~j~|~\mu (\Ff _{\ell_2}/\Ff_{\ell_2-1}) + \gamma < \mu (\Ff _{j}/\Ff _{j-1})\},$$
the map $\Phi$ induces a nonzero map $\overline{\Phi}: \Ee /\Ff _{\ell_1} \to \Ff _{\ell_2}\otimes \Tt_{\Dd}$. In particular, if $\ell_1 \ge \ell_2$, e.g. $s=2$ or $d_2/r_2+\gamma<d_s/r_s$, which imply $\ell_2=1$, then any such map $\Phi$ is $2$-nilpotent.


In \cite[Section 2]{BH} we consider the following exact sequence for $r\ge 2$
\begin{equation}\label{eqb1}
0\to \Oo_X^{\oplus (r-1)} \to \Ee \to \Ii_Z \otimes \Aa \to 0,
\end{equation}
where $\Aa$ is a line bundle of $\deg \Aa<0$ with $h^0(\Tt_{\Dd}\otimes \Aa^\vee)\ge r-1$ and $Z\subset X$ is a locally complete intersection of codimension two. Under certain assumptions on $Z$, we may choose $\Ee$ to be reflexive or locally free. Then any $(r-1)$-dimensional linear subspace of $H^0(\Tt_{\Dd}\otimes \Aa^\vee)$ produces a nonzero $2$-nilpotent co-Higgs field defined by the following composition:
$$\Ee \to \Ii_Z\otimes \Aa \to  \Tt_{\Dd}^{\oplus (r-1)} \to \Ee\otimes \Tt_{\Dd}.$$



Assume now the existence of an endomorphism $v: \Ee \rightarrow \Ee$ such that $v'\circ \Phi = \Phi \circ v$, where $v': \Ee \otimes \Tt_{\Dd}\rightarrow \Ee \otimes \Tt_{\Dd}$ is the induces by $v$ and the identity map on $\Tt_{\Dd}$. Since we assume that $\deg \Aa <0$, (\ref{eqb1}) is the Harder-Narasimhan filtration of $\Ee$. We also assume that (\ref{eqb1}) does not split and so every automorphism of $\Ee$ is induced by an element of $H^0(\Aa ^\vee )^{\oplus (r-1)}$ and an $(r-1){\times} (r-1)$-matrix of constants acting on $\Oo _X^{\oplus (r-1)}$. Note that, if $r=2$, these assumptions imply $h^0(\mathcal{E}nd (\Ee )) = 1+ h^0(\Aa ^\vee )$. In this case, the co-Higgs field $\Phi$ is obtained by composing a map $\Phi_1: \Ii _Z\otimes \Aa \rightarrow \Tt_{\Dd}$ with a map $\Phi_2: \Tt_{\Dd}\rightarrow \Ee \otimes \Tt_{\Dd}$ induced by the inclusion in (\ref{eqb1}).


\subsection{Case of $\Tt_{\Dd}$ not semistable}\label{us}
In this subsection we assume that $\Tt_{\Dd}$ is not semistable so that it admits the Harder-Narasimhan filtration 
\begin{equation}\label{thn}
\{0\}=\Hh_0 \subset \Hh_1 \subset \cdots \subset \Hh_h=\Tt_{\Dd}
\end{equation} 
with $h\ge 2$. Assume further that $\mu_+(\Tt_{\Dd})=\mu (\Hh_1)<0$. Since $h\ge 2$, we have $\dim X\ge h\ge 2$. 

Fix a torsion-free sheaf $\Ee$ of rank $r$ and degree $d$ with Harder-Narasimhan filtration (\ref{hn}). We assume the existence of a nonzero logarithmic co-Higgs field $\Phi : \Ee \rightarrow \Ee \otimes \Tt_{\Dd}$. 

\begin{lemma}\label{n0}
If $\Ee$ is reflexive, then $\Ff _i$ is also reflexive for each $i$.
\end{lemma}

\begin{proof}
In case $n=1$, the sheaf $\Ff _i$ in (\ref{hn}) is locally free and in particular reflexive. Now assume $n\ge 2$ and then we need to prove that $\Ff _i$ has depth at least two. This is true, because
$\Ee$ has depth at least two and $\Ee/\Ff _i$ has no torsion and so it has positive depth.
\end{proof}

\begin{remark}
Lemma \ref{n0} works for arbitrary $\Tt_{\Dd}$, even in the case $n=1$. 
\end{remark}

\begin{lemma}\label{n1}
We have $\Ff_1 \subseteq \mathrm{ker}(\Phi )$ and $s\ge 2$. 
\end{lemma}

\begin{proof}
Assume $\Phi (\Ff _1)\ne 0$ and let $i_0$ be the minimal integer $i\in \{1,\dots ,s\}$ such that $\Phi (\Ff _1) \subseteq \Ff _i\otimes \Tt_{\Dd}$. By the definition of $i_0$, the map $\Phi$ induces a nonzero map $\phi : \Ff _1 \rightarrow (\Ff _{i_0}/\Ff _{i_0-1})\otimes \Tt_{\Dd}$. Since the tensor product of two semistable sheaves, modulo its torsion, is again semistable by \cite[Theorem 2.5]{Maruyama} and $\mu (\Hh_1)<0$, the sheaf $gr((\Ff _{i_0}/\Ff _{i_0-1})\otimes \Tt_{\Dd} )$ given by the Harder-Narasimhan filtration of $\Tt_{\Dd}$ has all its factors with slope less than $\mu (\Ff _1)$. Thus we get $\Phi =0$, a contradiction.

Now $\Phi$ is a nonzero map with $\mathrm{ker}(\Phi )\supseteq \Ff _1$ and so we have $s\ge 2$.
\end{proof}

\begin{remark}\label{e1}
By Lemma \ref{n0}, the pair $(\Ff _1,0)$ is a logarithmic co-Higgs subsheaf of $(\Ee ,\Phi )$ and so $(\Ee ,\Phi )$ is not semistable. In particular, $\Ee$ is also not semistable. 
\end{remark}


\subsubsection{Rank $2$ case}
In this subsection we consider the co-Higgs sheaves $(\Ee, \Phi)$ with $\Ee$ reflexive of rank two and $\Phi$ nonzero. 

\begin{lemma}\label{n2}
If $\Ee$ is reflexive of rank two, then $\Phi$ is $2$-nilpotent.
\end{lemma}

\begin{proof}
Since $\Phi $ is nonzero, the sheaf $\Ff_1$ has rank one by Lemma \ref{n1}. Since $\Ff _1$ is reflexive on a smooth variety $X$ by Lemma \ref{n0}, it is a line bundle by \cite[Proposition 1.9]{Hartshorne1}. Now we get that $\Ee /\Ff _1\cong \Ii _Z\otimes \Aa$ for some line bundle $\Aa$ and some closed subscheme $Z\subset X$ with $\dim Z\le n-2$. By definition of Harder-Narasimhan filtration, we have $\deg \Aa <\deg \Ff _1$. Let $\psi : \Ee \rightarrow (\Ee /\Ff _1)\otimes \Tt_{\Dd}$ be the map induced by $\Phi$. Since $\mathrm{ker}(\Phi )\supseteq \Ff _1$ by Lemma \ref{n1}, it is sufficient to prove
that $\Phi (\Ee)\subseteq \Ff _1\otimes \Tt_{\Dd}$, i.e. $\psi =0$. Note that $\psi$ induces a map $\widetilde{\psi} :(\Ee /\Ff _1) \rightarrow (\Ee /\Ff _1)\otimes \Tt_{\Dd}$ with $\mathrm{Im}(\psi) =\mathrm{Im}(\widetilde{\psi})$, due to $\Ff_1 \subseteq\mathrm{ker}(\Phi )$. Since $(\Ee /\Ff _1)$ has rank one and it is torsion-free, it is semistable. Again as in the proof of Lemma \ref{n1}, since $\mu(\Hh_1) <0$ and the tensor product of two semistable sheaves, modulo its torsion, is semistable by \cite[Theorem 2.5]{Maruyama}, we have $\mu ((\Ee /\Ff _1)\otimes \Hh_1)< \mu (\Ee /\Ff _1)$ and so $\widetilde{\psi}=0$. Thus we have $\psi =0$.
\end{proof}

Now we describe all pairs $(\Ee ,\Phi)$ with $\Ee$ reflexive of rank two and $\Phi$ nonzero. By Lemma \ref{n1} and assumption that $\Phi$ is nonzero, the sheaf $\Ee$ is not semistable and $s=2$. By Lemmas \ref{n0},  \ref{n1}, \ref{n2} and \cite[Proposition 1.9]{Hartshorne1}, the map $\Phi$ is $2$-nilpotent and it fits into an exact sequence
$$0 \to \Ff_1 \to \Ee \to \Ii _Z\otimes \det (\Ee)\otimes \Ff_1^\vee \to 0,$$
with $Z$ a closed subscheme of $X$ with either $Z=\emptyset$ or $\dim Z=n-2$. Moreover, $\Phi$ is uniquely determined by a map $u: \det (\Ee)\otimes \Ff_1^\vee  \rightarrow \Ff_1 \otimes \Tt_{\Dd}$. Thus the set of all logarithmic co-Higgs structures on $\Ee$ is parametrized by 
$$V(\Ee):=H^0(\det (\Ee)^\vee \otimes \Ff_1^{\otimes 2} \otimes \Tt_{\Dd}).$$
The trivial element $0\in V(\Ee)$ corresponds to the trivial co-Higgs field $\Phi =0$. Note that $\Phi =0$ also exists for stable sheaves.

Now we reverse the construction. Fix two line bundles $\Ll_1$ and $\Ll_2$ on $X$ with $\deg \Ll_1 > \deg \Ll_2$ and a closed subscheme $Z\subset X$ such that a general extension
\begin{equation}\label{eqm1}
0 \to \Ll_1 \to \Ee \to \Ii _Z\otimes \Ll_2 \to 0
\end{equation}
is reflexive. We just observed that any co-Higgs field $\Phi : \Ee \rightarrow \Ee\otimes \Tt_{\Dd}$ is $2$-nilpotent and that $\Hom (\Ee ,\Ee \otimes \Tt_{\Dd}) \cong H^0(\Ll_1 \otimes \Ll_2 ^\vee \otimes \Tt_{\Dd})$. We may see \cite[Theorem 4.1]{Hartshorne1} for a description about the conditions on $\Ll_1$, $ \Ll_2$, $\omega _X$ and $Z$ assuring the existence of a reflexive sheaf fitting in (\ref{eqm1}) when $n=3$. Since (\ref{eqm1}) is the Harder-Narasimhan filtration of any $\Ee$ fitting into (\ref{eqm1}), so the family of the co-Higgs sheaves $(\Ee, \Phi)$ with $gr (\Ee)=\Ll_1\oplus (\Ii_Z\otimes \Ll_2)$ is parametrized by a fibration over $\PP \Ext_X^1(\Ii_Z\otimes \Ll_2, \Ll_1)$ whose fibre over $[\Ee]$ is $H^0(\Ll_1\otimes \Ll_2^\vee \otimes \Tt_{\Dd})$. 

\begin{remark}\label{s2}
Assume $s=2$ and $\mu_+(\Tt_{\Dd}) <0$. For a torsion-free coherent sheaf $\Ee$ of rank at least $2$, as in the proof of Lemma \ref{n2} we see that every logarithmic co-Higgs field $\Phi : \Ee \rightarrow \Ee \otimes \Tt_{\Dd}$ is integrable and $2$-nilpotent with
$$\Hom (\Ee,\Ee \otimes \Tt_{\Dd}) \cong \Hom (\Ee /\Ff _1,\Ff _1\otimes \Tt_{\Dd}),$$
where $\Ff _1$ is semistable, and $\Ee /\Ff _1$ is torsion-free and semistable. Recall that if $\Ee$ is reflexive, then so is $\Ff _1$ by Lemma \ref{n0}. Take any exact sequence
\begin{equation}\label{eqr1}
0 \to \Ff _1\to \Gg \to \Ee /\Ff _1\to 0.
\end{equation}
Any such extension in (\ref{eqr1}) is torsion-free. For any fixed $\Gg$ fitting into (\ref{eqr1}), not necessarily reflexive, the proof of Lemma \ref{n2} shows that every logarithmic co-Higgs field $\Phi : \Gg \rightarrow \Gg \otimes \Tt_{\Dd}$ is integrable and $2$-nilpotent with $\Hom (\Gg,
\Gg \otimes \Tt_{\Dd}) \cong \Hom (\Ee /\Ff _1,\Ff _1\otimes \Tt_{\Dd})$.
\end{remark}


\subsubsection{Rank $3$ case} 
We assume $r=3$ and that $\Ee$ is reflexive. Since we assume $\mu _+(\Tt_{\Dd}) <0$, we get $s\ge 2$ by Lemma \ref{n1} and so $s\in \{2,3\}$. 

\begin{remark}\label{r1}
The case $s=2$ is dealt in Remark \ref{s2}. In this case, the sheaf $\Ff_1$ is either a line bundle or a semistable reflexive sheaf of rank two with $\Ee/\Ff _1\cong \Ii _Z\otimes \Aa$ for some line bundle $\Aa$ and a closed subscheme $Z\subset X$ with $\dim Z\le n-2$. In both cases, we may apply Remark \ref{p1}. 
\end{remark}

From now on we assume $s=3$ and so the sheaf $\Ff _i$ in (\ref{hn}) has rank $i$ for each $i$. By Lemma \ref{n0}, the sheaf $\Ff_1$ is a line bundle and $\Ff _2$ is reflexive so that $\Ff _2/\Ff _1\cong \Ii _{Z_1}\otimes \Aa _1$ and $\Ee/\Ff _2\cong \Ii _{Z_2}\otimes \Aa _2$ with $\Aa _1,\Aa_2$ line bundles and $Z_1, Z_2$ closed subschemes of $X$ with dimension at most $n-2$. Here we have $\deg \Ff _1 > \deg \Aa _1 > \deg \Aa _2$. Set 
$$\delta(\Ee):=\mu_-(\Ff_2)-\mu_+(\Ee\otimes \Tt_{\Dd}),$$ 
where $\mu _-(\Ff _2) = \deg \Aa _1$ and $\mu _+(\Ee \otimes \Tt_{\Dd})= \deg \Ff _1+\mu _+(\Tt_{\Dd})$

\quad (a) Assume $\delta(\Ee)>0$ and then we have $\Phi _{|\Ff _2} =0$, i.e. $\Phi$ is uniquely induced by a map $u_1: \Ii _{Z_2}\otimes \Aa _2\rightarrow \Ee \otimes \Tt_{\Dd}$. Since $\Ee/\Ff_2 \cong \Ii_{Z_2}\otimes \Aa_2$ is of rank one and $\mu _+(\Tt_{\Dd}) < 0$, the composition of $u_1$ with the quotient map $\Ee \otimes \Tt_{\Dd}\rightarrow (\Ee /\Ff _2) \otimes \Tt_{\Dd}$ is trivial, implying $\mathrm{Im}(u_1)\subseteq \Ff _2\otimes \Tt_{\Dd}$. Thus $\Phi$ is uniquely determined by a map $u: \Ii _{Z_2}\otimes \Aa _2\rightarrow \Ff_2 \otimes \Tt_{\Dd}$. Conversely, any map $u: \Ii _{Z_2}\otimes \Aa _2\rightarrow \Ff_2 \otimes \Tt_{\Dd}$ induces a $2$-nilpotent logarithmic co-Higgs field on $\Ee$ by taking the composition $u\circ \pi$, where $\pi : \Ee \rightarrow \Ee /\Ff _2$ is the quotient map.

\quad (b) Assume now $\delta(\Ee)\le 0$. Set $\Bb := \mathrm{Im}(\Phi _{|\Ff _2})$ and $\Gg := \mathrm{Im}(\Phi )$. Since we have
$$\mu _+((\Ee /\Ff _2)\otimes \Tt_{\Dd}) = \mu (\Ee /\Ff _2) +\mu _+(\Tt_{\Dd}) < \mu (\Ee /\Ff _2),$$
the composition of $\Phi$ with the quotient map $\Ee \otimes \Tt_{\Dd} \rightarrow (\Ee /\Ff _2) \otimes \Tt_{\Dd}$ is trivial and so we have $\Gg \subseteq \Ff _2\otimes \Tt_{\Dd}$. If $\Bb $ is trivial, then we may apply part (a), i.e. $\Phi$ is $2$-nilpotent and it is uniquely induced by $u: \Ii _{Z_2}\otimes \Aa _2\rightarrow \Ff_2 \otimes \Tt_{\Dd}$. Now we assume that $\Bb$ is not trivial. Since $\Phi (\Ff _1) =0$ and $\Ff _2/\Ff _1$ is a torsion-free sheaf of rank one, we have $\Bb \cong \Ff _2/\Ff _1$ and so $\mathrm{rk}(\Gg)\in \{1,2\}$. Note that we have $\Bb \subseteq \Ff _1\otimes \Tt_{\Dd}$ from $\mu _+(\Ee /\Ff _2)+\mu _+(\Tt_{\Dd})  < \mu _+(\Ee/\Ff _2)$. 

 \quad (b-i) First assume $\mathrm{rk} (\Gg) =1$ and then $\Bb$ is a subsheaf of $\Gg$ with the same rank. Since $\Ff _1\otimes \Tt_{\Dd}$ is a saturated subsheaf of $\Ff _2\otimes \Tt_{\Dd}$, we have $\Gg \subseteq \Ff _1\otimes \Tt_{\Dd}$. Thus $\Phi$ is uniquely determined by a map $\Ee /\Ff _1\rightarrow \Ff _1\otimes \Tt_{\Dd}$, i.e. by an element of $H^0(\Tt_{\Dd}\otimes \Ff _1 \otimes \Aa ^\vee )$; the converse also holds, but we cannot guarantee the integrability of the associated logarithmic co-Higgs field. 
 
 \quad (b-ii) Now assume $\mathrm{rk}(\Gg) =2$. Since we have $\Gg=\psi (\Ee /\Ff _1)$ for the map $\psi : \Ee/\Ff _1 \rightarrow \Ee \otimes \Tt_{\Dd}$, the map $\psi$ is injective as a map of sheaves and $\Gg \cong \Ee /\Ff _1$. In this case we also have $\Ff _1 =\mathrm{ker}(\Phi )$. We get that $\Ee$ is a reflexive sheaf fitting into an exact sequence
\begin{equation}\label{eqr2}
0\to \Ff _1\to \Ee \stackrel{f}{\to} \Gg \to 0
\end{equation}
with $\Ff _1$ a line bundle and $\Gg$ a torsion-free unstable sheaf of rank two with $\deg \Ff _1>\mu _+(\Gg)$. The map $\Phi$ is determined by a unique injective map $v: \Gg \rightarrow \Ff _2\otimes \Tt_{\Dd}$. Conversely, set $\Gg _1 \subset \Gg$ to be the Harder-Narasimhan filtration of $\Gg$ and $\Ff _2 = f^{-1}(\Gg _1)$, where $f$ is the surjection in (\ref{eqr2}). Then the composition of the quotient map $\Ee \rightarrow \Ee /\Ff _1$ with an injective map $\Gg \rightarrow \Ff _2\otimes \Tt_{\Dd}$ induces a logarithmic co-Higgs field $\Phi$ with the given data $(\Ff _1, \Ff _2, \Gg)$, which does not necessarily satisfy the integrability condition. Note that if $\Gg \subset \Ff _1\otimes \Tt_{\Dd}$, i.e. $\Phi$ comes from an injective map $\Gg \rightarrow \Ff _1\otimes \Tt_{\Dd}$, then $\Phi$ is $2$-nilpotent and so it is integrable.

\begin{example}
Assume that $\Tt _{\Dd}$ is not semistable with Harder-Narasimhan filtration (\ref{thn}) and set $\mu _2(\Tt _{\Dd}):= \mu (\Hh _2/\Hh _1)$. Let $\Ee$ be a torsion-free sheaf of rank $r$ with (\ref{hn}) as its Harder-Narasimhan filtration and assume $\mu_+(\Ee)-\mu_-(\Ee) < \mu _2(\Tt_{\Dd} )$. In this case, for any map $\Phi : \Ee \rightarrow \Ee \otimes \Tt_{\Dd}$, the sheaf $\mathrm{Im}(\Phi)$ is contained in the subsheaf $\overline{\Ee\otimes \Hh_1}$ of $\Ee\otimes \Tt _{\Dd}$, which is the image of the natural map $\Ee \otimes \Hh _1\rightarrow  \Ee \otimes \Tt _{\Dd}$. We have $\overline{\Ee\otimes \Hh _1} \cong \Ee \otimes \Hh _1$ if either $\Ee$ or $\Hh_1$ is locally free. Note that $\Hh _1$ is locally free, if it has rank one, because $\Hh _1$ is reflexive and $X$ is smooth; see \cite[Proposition 1.9]{Hartshorne1}. In particular, if $n=2$, then $\Hh _1$ is a line bundle and $\mu _2(\Tt _{\Dd}) =\mu _-(\Tt _{\Dd})$. Thus under these assumptions we may repeat the observations given in the case $\Tt _{\Dd}$ semistable using $\Hh _1$ instead of $\Tt _{\Dd}$. Without any assumption on $\mu _2(\Tt _{\Dd})$ we may see at least a part of the logarithmic co-Higgs fields of $\Ee$ in this way.
\end{example}

\providecommand{\bysame}{\leavevmode\hbox to3em{\hrulefill}\thinspace}
\providecommand{\MR}{\relax\ifhmode\unskip\space\fi MR }
\providecommand{\MRhref}[2]{%
  \href{http://www.ams.org/mathscinet-getitem?mr=#1}{#2}
}
\providecommand{\href}[2]{#2}

\end{document}